\documentclass[11pt]{amsart}

\date{\today}
\usepackage{xcite}

\usepackage{hhline}
\usepackage{cancel} %
\usepackage{amsmath}
\usepackage{marginnote}  %
\usepackage{amssymb}
\usepackage{amsthm}
\usepackage{accents}
\usepackage{amsfonts}
\usepackage{mathtools}
\usepackage{enumerate}                 
\usepackage{comment}
\usepackage{hyperref}
\usepackage{multirow}

\usepackage[dvipsnames]{xcolor}
\usepackage{soul}
\usepackage{enumitem}
\setlist[enumerate]{%
wide =0.5\parindent,
listparindent=0pt%
}%
\allowdisplaybreaks
\mathtoolsset{showonlyrefs}

\let \i\undefined
\let \div \undefined
\newcommand{\Id}{\text{\rm Id}}

\DeclareMathOperator{\curl}{curl}

\DeclareMathOperator{\div}{div}

\DeclareMathOperator{\supp}{supp}

\DeclareMathOperator{\diag}{diag}

\newcommand{\i}{\mathrm{i}}

\newcommand{\R}{\mathbb{R}}
\newcommand{\B}{\mathbb{B}}
\newcommand{\C}{\mathbb{C}}

\newcommand{\p}{\partial}
\newcommand{\n}{\nabla}

\renewcommand{\d}{{\rm d}}
\let\epsilon\varepsilon
\newcommand{\e}{\epsilon}

\renewcommand{\r}{\rho}

\newtheorem{lemma}{{Lemma}}
\newtheorem{theorem}{Theorem}
\newtheorem{proposition}{Proposition}

\theoremstyle{remark}
\newtheorem{remark}{Remark}
\numberwithin{equation}{section}

\def \u#1{\underline{#1}}
\let \o \undefined
\def \o#1{#1}

\newcommand{\be}[1]{\begin{equation}\label{#1}}
\newcommand{\ee}{\end{equation}}
\renewcommand{\r}[1]{\eqref{#1}}

\title{The DC Kerr Effect in Nonlinear Optics}

\author[N. Eptaminitakis] {Nikolas Eptaminitakis}
\address{Institut für Differentialgeometrie, Leibniz Universit\"at Hannover,
	Welfengarten 1, 30167 Hannover, Germany}

\author[P. Stefanov]{Plamen Stefanov}
\address{Department of Mathematics, Purdue University, West Lafayette, IN 47907}
\thanks{P.S. partly supported by  NSF  Grant DMS-2154489. N.E. acknowledges support from the Graduate Academy of Leibniz University Hannover}

\begin{document}

\begin{abstract}
We use weakly nonlinear geometric optics to study a model for the  DC Kerr effect (the Kerr electro-optic effect), in which a light beam propagating through a material with strong nonlinear optical properties can have its polarization rotated by applying a strong external electric field. This effect is used to build fast switches (Kerr cells). We prove existence of an exact solution of the nonlinear Maxwell system with a cubic Kerr nonlinearity, with the wavelength $h$ being a small parameter. We justify the effect within this model, and also solve the inverse problem of recovery of the nonlinear susceptibility $\o{\chi}^{(3)}$ from the change of the polarization.  
\end{abstract} 
\maketitle

\section{Introduction}
The goal of this work is to analyze a model for the DC Kerr effect in nonlinear optics (also known as the Kerr electro-optic effect),  and to solve the inverse problem of recovering the relevant nonlinearity  parameter. We use the so-called weakly nonlinear geometric optics with the wave length proportional to a small parameter $h>0$. 
The DC Kerr effect can be described as follows, see, e.g., \cite{new2011-non-linear, boyd2020nonlinear}. A polarized beam of light is sent through a medium for which the polarization $P$ exhibits non-negligible nonlinear dependence on the electric field $E$ which is of Kerr type, that is, the nonlinearity is  cubic, proportional to $|E|^2E$, and its quadratic part vanishes  due to cancellations. Such materials include liquids, gases, amorphous solids, and even some crystals.  The intensity of the light is not strong enough for it to be affected significantly by  nonlinear interaction. When a strong constant (or slowly varying) electric field $E_0$ is applied perpendicularly to the direction of propagation, the nonlinear  interaction becomes significant and it creates birefringence causing the polarization to rotate along the way. This effect is used to create very fast switches, called Kerr cells, see Section~\ref{sec_cell}, which have been used to measure the speed of light. It can also be used to measure the nonlinearity of the material, see \cite{al2014determination}, or the strong field intensity. 

To model this effect, we work with Maxwell's equations in $\R_t\times \R_x^3$ in the absence of free currents and charges (see \cite[Sec. 2.1]{boyd2020nonlinear})
\begin{subequations}
\begin{align}
		\p_t D-\curl H&= 0\label{eq1}\noeqref{eq1},\\
		\p_t B+\curl E&=0\label{eq2}\noeqref{eq2},\\
		\div B&=0\label{eq3}\noeqref{eq3},\\
		\div D&=0\label{eq4}\noeqref{eq4},
\intertext{with the constitutive equations}	
B=\mu_0 H,\qquad D=\e_0 E+P&, \qquad P=\e_0\chi^{(1)}E+\e_0\o{\chi}^{(3)}|E|^2E.\label{eq5}
\end{align}
\end{subequations}
As already mentioned, $E$ is the electric field and is $D$ the electric displacement field.
Moreover, $B$ and $H$ represent the magnetic flux density and magnetic field strength respectively,  whereas
 $\epsilon_0$ and $\mu_0$  are the electric conductivity and the magnetic permeability in free space respectively. 
Our assumption \eqref{eq5} on the polarization density $P$ indicates instant but nonlinear (third order) polarization.
We assume that the  linear susceptibility $\chi^{(1)}$ and the third order nonlinear susceptibility $\o{\chi}^{(3)}$ are real valued and compactly supported (we will  take $\chi^{(1)}=0$ later). 

In our first main result, Theorem~\ref{thm_main} below, we construct solutions depending on the small parameter $h>0$, with a suitable $h$-dependent scaling of the strong electric field and the beam, see \r{A1}. The  weakly geometric optics originated in the physics literature and was developed in the mathematical one in \cite{Metivier-Notes,Metivier-Joly-Rauch, Joly-Rauch_just, Donnat-Rauch_dispersive, Dumas_Nonlinear-Geom-Optics, JMR-95, Rauch-geometric-optics} and other works.  We derive the $h$-dependent scaling needed to explain the rotation of the polarization effect within this model, and then justify it. Then in Theorem~\ref{thm_inverse}, we show that one can recover the X-ray transform of $\chi^{(3)}$ along the light rays from the change of the polarization, and then $\chi^{(3)}$ itself.

There are a few related mathematical works  about inverse problems for nonlinear optics that we want to mention. A scalar stationary model has been studied in \cite{John_Kerr}. In  \cite{Assylbekov-Ting3}, a recovery of $\chi^{(3)}$ is proven with a stationary model but that model excludes formation of harmonics, which actually are a part of nonlinear optics. The authors study a fixed frequency with the higher linearization method (see next paragraph). A stationary model with quadratic nonlinearity is considered in \cite{Assylbekov-Ting2}, approximating the nonlinear Maxwell system to concentrate on second order harmonics only.  
The inverse problem in nonlinear acoustics has been the focus of \cite{acosta2021nonlinear, Uhlmann-Zhang-acoustics, Kaltenbacher_2021, S_N-Westervelt}. That model is  close to a scalar version of nonlinear optics with a quadratic nonlinearity $\chi^{(2)}$. 

Most of the inverse problems results for nonlinear hyperbolic PDEs are based on the higher order linearization  method  pioneered in \cite{KLU-18} and \cite{LassasUW_2016}. The idea of the method is to consider small solutions with one or several small parameters and take the Taylor expansion of the solution with respect to them. The information about the nonlinearities is contained in those  higher order terms,   which solve the linearized PDE with source terms.  Other works in this direction are 
\cite{Hintz1,Hintz2,LassasUW_2016, LUW1, lassas2020uniqueness, Hintz-U-19, uhlmann-zhang-2021inverse, OSSU-principal}. On the other hand, inverse problems for semilinear wave type PDEs  were studied using nonlinear geometric optics  by the second author and S\'a Barreto in \cite{S-Antonio-nonlinear, S-Antonio-nonlinear2}, and for the quasilinear Westervelt equation by the present authors in \cite{S_N-Westervelt}, in regimes where the nonlinearity affects the principal term in \cite{S-Antonio-nonlinear, S_N-Westervelt} and the sub-principal one in \cite{S-Antonio-nonlinear2}, because of the nature of the nonlinear problem there. In particular, we justify in \cite{S-Antonio-nonlinear, S_N-Westervelt}  nonlinear effects observed in physics. They happen for highly oscillatory solutions (in some frequency band), and the effect is in the leading order term. One of the difficulties with this approach is that a priori solvability of the corresponding PDE is not guaranteed by the ``small initial conditions'' theorems. The solutions of interest are not small in the required Sobolev norms because of the high oscillations. On the other hand, they are physical, and they should exist if the model is good.

\section{Main results} \label{sec:main_results}
We work in a large ball $B(0,R_0) \subset\R^3$, $R_0\gg1$. Assume  $\chi^{(1)}=0$ (constant speed), and $\chi^{(3)}\in C_0^\infty (\R^3)$, with $\supp \chi^{(3)}\subset B(0,R)$, $0<R<R_0$. 
Assume also $\e_0\mu_0=1$;
 this amounts to rescaling the time variable $t$ to $\tilde t = (\e_0\mu_0)^{-1/2}t$, so that the speed of light $c=(\e_0\mu_0)^{-1/2}$ becomes one. 
Under those assumptions and as explained in Section \ref{sec:reduction}, we can convert \eqref{eq1}--\eqref{eq5} to the second order $3\times3$ hyperbolic system %
\begin{equation}\label{eq:second_order_metric3}
	\p_t^2 E- \Delta E + \nabla \div E=-\chi^{(3)}\p_t^2 (|E|^2E), 
\end{equation}
under the divergence free condition 
\begin{equation}\label{div0}
  \div \big( E  + \o\chi^{(3)}|E|^2E \big)=0.
\end{equation}
We prove that this is equivalent to the Maxwell system in Proposition~\ref{pr1} (with suitable initial conditions). 
In \eqref{eq:second_order_metric3} and throughout, $\nabla$, $\div$, and $\Delta$ respectively denote the gradient, divergence, and component-wise negative Laplacian in the space variables only. %

For a fixed $\omega \in S^2$, we  are interested in a solution of the form
\begin{equation}   \label{A1}
	E = h^{1/2}\underline{E}(x,h)  + h^{3/2} U(t,x,\phi/h,h),\qquad \phi\coloneqq -t+ x\cdot\omega,
\end{equation}with 
\begin{equation}\label{U:exp}
	U (t,x,\theta,h) \sim U_0(t,x,\theta) + hU_1(t,x,\theta)+\dots,
\end{equation}
where all terms are $2\pi$-periodic in $\theta$.
The choice of the scaling is explained in Section~\ref{sec_gen}; in short, it is the one guaranteeing the effect we want to model. We think of 
$h^{1/2}\underline{E}(x,h)$ as the ``strong'' electric field. It is indeed strong relative to the ``beam'' $h^{3/2}U$. The $\sim h^{1/2}$ magnitude in absolute terms is not an indicator of its strength, since we can rescale $E$, $U$, and $\chi^{(3)}$ by suitable powers of $h$, and change the power of $h$ multiplying $\underline{E}(x,h)$, see Section~\ref{sec_scale}. 
The stationary field $h^{1/2}\underline{E}$ should have an expansion 
\begin{equation}   \label{E0}
	h^{1/2}\underline{E}(x,h) \sim h^{1/2}\left(  \u{E}_0(x)+ h \u{E}_1(x) +\dots\right),
\end{equation}
and satisfy  \r{div0}.
We take $\underline{E}_0=\textrm{const.}$ in Theorem~\ref{thm_main}\ref{item:c} below but keep it variable in part \ref{item:a}.

We can impose a suitable initial condition that will generate a single-phase solution of the form \eqref{A1}.
However it turns out to be more convenient  to instead impose initial conditions of the form 
\begin{equation}   \label{IC1}
	E|_{t=0} =  E_{(0)}\coloneqq   h^{1/2}\underline{E}(x,h)+ 2h^{3/2}   U_{\rm init}(x) \cos \frac{x\cdot\omega}{h}, \quad  %
	\partial_t E|_{t=0}=0,
\end{equation}
where   $U_{\rm init}\in C_0^{\infty}(\R^3;\, \R^3)$ is supported away from $B(0,R)$ and satisfies \begin{equation}   \label{IC0}
\div  U_{\rm init}(x)=0,\quad   \omega\cdot   U_{\rm init}(x)=0,
\end{equation} 
which, as we will see, guarantee the divergence free conditions \eqref{eq3}, \eqref{eq4}.
The initial condition \eqref{IC1} will actually generate \emph{two} waves at the level of geometric optics, propagating in opposite directions: one towards the region of interest $B(0,R)$  with phase $\phi_{\rm in}=\phi=-t+x\cdot \omega$, and one away from it, with phase $\phi_{\rm out}=t+x\cdot \omega$ (see Figure~\ref{fig:the_setup}).
This is explained in more detail in Section \ref{sub:initialization_of_the_problem_the_linear_solution}. 
Then \eqref{U:exp} holds in $B(0,R)$ for $t\geq 0$. When $\underline{E}_0=\textrm{const.}$, 
we define the phase retardation
\begin{equation}   \label{tau}
	\tau(x) =  \frac12|\u{E}_0|^2 \int_{-\infty}^0 \chi^{(3)}(x+\sigma \omega )\,\d\sigma.
\end{equation}

\begin{figure}[ht]
	\includegraphics[scale=.6]{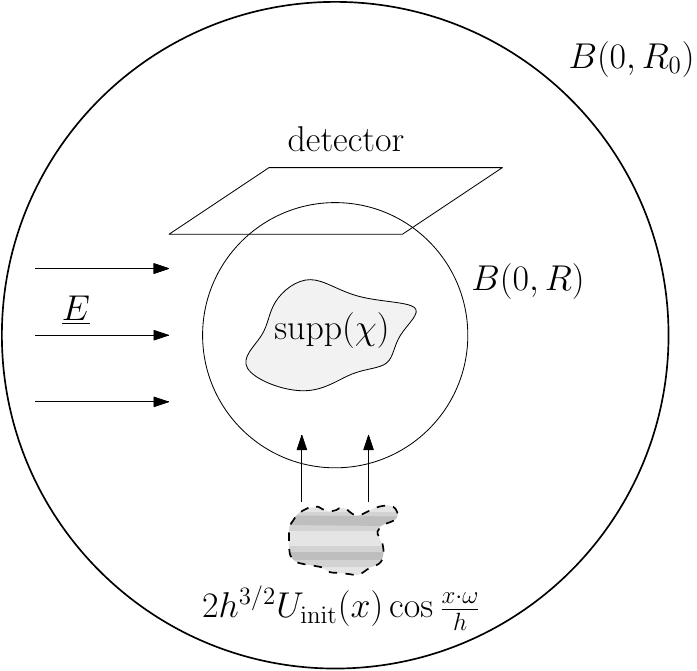}
	\caption{\small The setup. The backward (downward) propagating free solution is not shown.}
	\label{fig:the_setup}
\end{figure}

In the next theorem we formulate a lighter version of our main results, with remarks about the general ones. 
We assume  without loss of generality that $\omega=e_1=(1,0,0)$; then  $U_{{\rm init},1}=0$. We denote by $\Delta_D$ the Dirichlet Laplacian in $B(0,R_0)$; the  spaces $\mathcal{B}_\sigma^{m}$ are defined in Section \ref{sec:exact_solutions}.

\begin{theorem}  \label{thm_main} Under the assumptions above, we have the following. 
\begin{enumerate}[label = (\alph*)]
  \item\label{item:a} \emph{[Existence of a strong electric field.]}
Let $f\in H^{s-1/2}(\partial B(0,R_0))$, where $s>5/2$, and let ${u_f}\in H^{s}(B(0,R_0))$ be its harmonic extension  to $B(0,R_0).$ 
Also fix $\epsilon>0$.
There exists $h_0>0$ such that for $0< h\leq h_0$ there exists a unique stationary solution $h^{1/2}\underline{E}(\cdot,h)\in H^{s-1}(B(0,R_0))$ of \eqref{eq:second_order_metric3} 
and \eqref{div0} 
 satisfying $\underline{E}=\nabla {u_f}$ on $\partial B(0,R_0)$, $\curl E=0$ in $B(0,R_0)$, 
 and $\|E-\nabla {u_f}\|_{H^{s-1}(M)}\leq \epsilon$. %
It admits an asymptotic expansion in $h$ of the form
\begin{equation}
 	h^{1/2}\underline{E}=h^{1/2}\nabla{u_f}-h^{3/2} \nabla \Delta_D^{-1}\div(\chi |\nabla {u_f}|^2\nabla {u_f})+O_{H^{s-1}}(h^{5/2}).
 \end{equation}

\item\label{item:b} \emph{[The beam with no strong electric field.]} Let $m\geq 3$ be an integer and fix $T>0$.
There exists $\sigma>0$ and $h_0>0$ such that for  $0<h\leq h_0$ there exists a solution of \eqref{eq:second_order_metric3} defined for $t\in [0,T]$, subject to initial conditions \eqref{IC1} with $\underline E=0$, and satisfying in $B(0,R)$%
\begin{equation}\label{eq:exp_no_nonlinearity}
	E_{\rm in} =  
	    h^{\frac32} \Big(0, U_{{\rm init},2}(x-t e_1) \cos \frac{-t+x^1}{h}  ,  U_{{\rm init},3} (x-te_1)  \cos \frac{-t+x^1}{h}  \Big)+O(h^\frac52),
\end{equation}
where the error is  term is the restriction to $[0,T]\times B(0,R)$ of an element in $h^{5/2}\mathcal{B}_\sigma^{m}$.

\item\label{item:c} \emph{[The beam in presence of a strong electric field.] }
Let $m\geq 3$ be an integer and fix $T>0$. Construct $\underline{E}$ as in \ref{item:a} with $s\geq m+1$ and $f=|E_0|x_3$, where $|\underline{E}_0|=\textrm{const.}$, so that  $\underline{E}= |\underline{E}_0| e_3+O(h)$. 
There exists $\sigma>0$ and $h_0>0$ such that for  $0<h\leq h_0$  there exists a solution of \eqref{eq:second_order_metric3} defined for $t\in [0,T]$ with initial conditions \eqref{IC1}, having an expansion of the form \eqref{U:exp} in $B(0,R)$ with its first terms being  
\begin{equation}
	\begin{aligned}\label{exp1}
		E_{\rm in} &=h^\frac12 |\u{E}_0|e_3 \\
		&\quad +  h^{\frac32} \Big(0, U_{{\rm init},2}(x-t e_1) \cos\Big(\frac{-t+x^1}{h}+\tau(x)\Big),  U_{{\rm init},3} (x-te_1)  \cos\Big( \frac{-t+x^1}{h}+3\tau(x)\Big) \Big)\\
		 & \qquad + O(h^{\frac52}), 
	\end{aligned} 	
\end{equation} 
where the error is  term is the restriction to $[0,T]\times B(0,R)$ of an element in $h^{5/2}\mathcal{B}_\sigma^{m}$.
In \eqref{exp1}, $\tau$ is as in \eqref{tau}.
\end{enumerate}
In parts \ref{item:b} and \ref{item:c}, the solution is unique among functions in a sufficiently small neighborhood of an asymptotic solution with finitely many terms, see Proposition \ref{prop_Gues} for a precise statement. 
\end{theorem}

\noindent Note that all statements above refer only to behavior of solutions in the region of interest $B (0,R)$, whereas in Lemma \ref{lm:stationary} and Proposition~\ref{prop_Gues} later we make precise statements about their global behavior.  

Part \ref{item:a} states the existence of a stationary electric field, which is a differential of a harmonic function up to  lower order terms. A stationary solution of \eqref{eq:second_order_metric3}--\eqref{div0}  as in the theorem that is also curl-free yields, together with a stationary $H$, a stationary solution to \eqref{eq1}--\r{eq5}.  
To get a constant $\underline{E}_0$ as in part \ref{item:c}, we chose $f=|\u E_0|x_3$ with $|\u E_0|$ constant. 
In fact, a constant $\underline E$ does not satisfy the divergence free condition for $D$, unless $\chi^{(3)}$ is constant. The proof  reduces to solving a nonlinear second order elliptic PDE, see \eqref{eq:el} in Section~\ref{eq:choosing_the_strong_electric_field}. %
Part \ref{item:b}  says that under the chosen prefactor of $h^{3/2}$, \textit{and no strong electric field}, the principal term of the beam would the same as if the system were linear, i.e.,  $\chi^{(3)}=0$.
In part \ref{item:c}, the strong field is on, and the beam is interacting with it. This creates phase shifts, and a birefringent effect rotating the polarization, which we explain in more detail in Sections~\ref{sub:the_asymptotic_solution_the_leading_term}, and \ref{sec_cell}.

The proof of existence of (exact) solutions close to the  asymptotic ones constructed in \ref{item:c} is non-trivial, since we lack a priori existence, uniqueness, and well-posedness results for the nonlinear Maxwell system. We use a result of Guès (\cite{Gues93}) to prove existence of such solutions,
at least in our region of interest, thus justifying the formal asymptotic expansions. 
Rigorous justifications of asymptotic expansions for solutions of quasilinear PDE have also been provided by other authors under a variety of assumptions, see e.g., \cite{Kalyakin1989}, \cite{Yoshikawa1993},	  \cite{Jol-Met-Rau-just-Duke}, and the references there.
Some of the above works also study multi-phase asymptotic solutions, in which case  waves corresponding to different phases can interact and form new ones that propagate in new directions;
even though our asymptotic solution exhibits two phases, one does not interact with the nonlinearity at the level of geometric optics, and the construction of an asymptotic solution is similar to a single-phase one.

Our analysis confirms, at least in principle, formation of harmonics, typical for nonlinear problems of wave type. 
In the principal term for the beam, of order $O(h^{3/2})$, we observe only the modes $\pm1 $ (this is due to our choice of initial data).
The next term, of order $O(h^{5/2})$, can have the harmonics $k=-2,-1,1,2$, see Lemma \ref{lm:U1}. They come from interactions of two of the harmonics in the previous term with one instance of the strong electric field, via the cubic nonlinearity. 
At order $O(h^{7/2})$ we obtain harmonics corresponding to $\pm 1$, $\pm 2$, $\pm 3$, as well as a zero harmonic that solves \eqref{eq:first_zero_harm}.

As an immediate corollary of Theorem~\ref{thm_main}, we can recover the X-ray transform of $\chi^{(3)}$ along the lines through its support. We vary $\omega$ now, and assume that $U_{\rm init}$ varies as well, so that the rays through its support in the direction $\omega$ cover $B(0,R)$. This is equivalent to illuminating $B(0,R)$ from all directions with a beam wide enough to cover it. For a fixed $\omega$, we may have several beams covering $B(0,R)$ as well. We place a ``detector'' at the plane $\pi_{R,\omega} := \{x|\;x\cdot\omega=R\}$, see Figure~\ref{fig:the_setup}, and we assume $R_0\gg1$ so that the ball $B(0,R)$, projected to $\pi_{R,\omega}$ fits in $B(0,R_0/2)$. 

\begin{theorem}   \label{thm_inverse}
Under the assumptions of Theorem~\ref{thm_main}\ref{item:c}, for every $\omega\in S^2$, let  $U_{\rm init}\not\equiv0$ be such that all rays in the direction of $\omega$ issued from $\supp U_{{\rm init},2} \cap \supp U_{{\rm init},3}$ cover $B(0,R)$ and eventually exit it.   Then $E_\textrm{\rm in}(t,x)$ known for $x\in \pi_{R,\omega}$ and for $t\in [0,2R_0]$, and all $0<h\ll1$,  recovers the X-ray transform of $\chi^{(3)}$ in the direction $\omega$ uniquely. Varying $\omega$,  $\chi^{(3)}$ can be recovered uniquely as well. 
\end{theorem}

\noindent
In fact, the proof shows that knowing $E_\textrm{\rm in}(t,x)$ for $(t,x)$ as stated and  for $0<h\ll1$, we can recover the X-ray transform of $\chi^{(3)}$ up to $O(h)$. 

The article is organized as follows. In Section~\ref{sec:reduction}, we reduce the Maxwell system \r{eq1}--\r{eq5} to the wave type of equation \r{IC_E1} with appropriate divergence free conditions. We build the weakly geometric optics solution in Section~\ref{:the_weakly_geometric_optics}. In particular, in Section~\ref{sub:the_asymptotic_solution_the_leading_term} we demonstrate the rotation of the polarization effect presented in the leading term of the oscillatory expansion. In Section~\ref{sec_cell}, we show how Kerr cells fit in our analysis.     The strong electric field $\underline E$ is constructed in Section~\ref{eq:choosing_the_strong_electric_field} by solving a nonlinear elliptic PDE of double-phase type.  Some technical proofs showing the derivation of the profile equations can be found in Section~\ref{sec:the_profile_equations_proofs_of_lemmas_ref_lm_eikonal_transport}. In Section~\ref{sec:exact_solutions}, we show that there exist exact solutions close to the asymptotic ones, using a result by Guès.
Finally,  Theorem~\ref{thm_inverse} is proved in Section~\ref{sec_inverse}.

\textbf{Acknowledgments.} We  thank Arshak Petrosyan for pointing out to us references  \cite{Lad-Ural} and \cite{BaroniColomboMingione}, and for the discussion  on nonlinear elliptic PDEs. 

\section{Reduction to a 2nd order hyperbolic system}\label{sec:reduction}
In this section we show a  correspondence between solutions to \eqref{eq1}--\eqref{eq5} and \eqref{eq:second_order_metric3}--\eqref{div0}.
Taking divergence of \eqref{eq1} and \eqref{eq2}, we see that $\div B$ and $\div D$ remain constant in time. Therefore, equations \eqref{eq3}, \eqref{eq4} can be considered as divergence free requirements on the initial conditions that remain preserved for all $t$. 
We can eliminate $B$ and $D$ to obtain 
\begin{subequations}
\begin{align}
		\p_t (\e_0 E+P)-\curl H&=0, \label{eq1a}\noeqref{eq1a}\\
		\p_t \mu_0 H+\curl E&=0, \label{eq2a} \\
		\div \mu_0 H&=0,\label{eq3a}\noeqref{eq3a}\\ \div  (\e_0 E+P )&=0, \label{eq4a}\noeqref{eq4a}
\end{align}
\end{subequations}
where we view $P=P(E)$ as given by \eqref{eq5}. 
Equation \r{eq4a} has the form
\be{div}
\div D=0\quad   \Longleftrightarrow\quad   \div \Big( \big(1+ \chi^{(1)}\big)E  + \o\chi^{(3)}|E|^2E \Big)=0,
\ee 
and assume for now $1+ \chi^{(1)}>0$, although eventually we take $\chi^{(1)}=0$. 
We impose initial conditions  
\begin{equation}   \label{IC_EH}
(E,H)|_{t=t_0} = (E_{(0)},H_{(0)}),
\end{equation}
which cannot be arbitrary; the 
corresponding initial conditions $(D_{(0)},B_{(0)})$ for $(D,B)$ must be divergence free. This means that
\begin{equation}  \label{IC_E1}
	E_{(0)}+ \chi^{(1)}E_{(0)}+ \chi^{(3)}|E_{(0)}|^2E_{(0)}      = \e_0^{-1} D_{(0)}
\end{equation}
with $\div D_{(0)}=0$, as in \eqref{div}, and $\div  H_{(0)}=0$. Given divergence free $(D_{(0)},B_{(0)})$, we can easily solve for $ H_{(0)}=\mu_0^{-1} B_{(0)}$, and \r{IC_E1} can be solved for $E_{(0)}$ when $\chi^{(3)} |E_{(0)}|^2$ is small enough, picking up the solution close to $ H_{(0)}$.  This would be the case below since $E\sim h^{1/2}$.

We can transform \eqref{eq1a}--\eqref{eq2a} into a second order system for the electric field, of the form
\begin{equation}\label{eq:wave1}
	\e_0(1+\chi^{(1)})\p_t^2 E+\mu_0^{-1}\curl \curl E=-\e_0\chi^{(3)}\p_t^2 (|E|^2E).
\end{equation}
We need initial conditions
\begin{equation}   \label{IC_wave}
	E|_{t=t_0}= E_{(0)}, \quad \p_t E|_{t=t_0} =  E_{(0)}',
\end{equation}
and we use a similar notation $ D_{(0)}':= \p_t D|_{t=t_0}  $ below. 

Every solution $(E,H)$ of \eqref{eq1a}--\eqref{eq4a} with initial conditions \eqref{IC_EH} satisfying the divergence free conditions above  provides a solution $E$ of \eqref{eq:wave1} with initial conditions \eqref{IC_wave}. The second initial condition  $E_{(0)}'$ in \r{IC_wave} can be found from 
\begin{equation}
D_{(0)}'=	\e_0 \p_t   \big( E+ \chi^{(1)}E+ \chi^{(3)}|E|^2E   \big)|_{t=t_0}  = \curl H_{(0)}.  \label{IC_E2}
\end{equation} 
Equation \r{IC_E2} can be solved for $E_{(0)}'$ under the same conditions on $\chi^{(1)}$, $\chi^{(3)}$ and $E_{(0)}$. 
Then the so found $(E_{(0)}, \partial_t E_{(0)})$ yield the divergence free conditions $\div  D_{(0)}=0$ (by assumption), and  $\div  D'_{(0)}=0$ (by \eqref{IC_E2}). 

Conversely, assume we are given a solution $E$ to \eqref{eq:wave1} with initial conditions \eqref{IC_wave} for which  $\div D_{(0)}=\div D_{(0)}'=0$. 
We show first how to choose a divergence free $H_{(0)}$ solving \eqref{IC_E2}. 
We are looking for  $H_{(0)}$ satisfying  
\begin{equation}   \label{EH0}
\div H_{(0)}=0, \quad \curl H_{(0)} = D_{(0)}'.
\end{equation}
with $ D_{(0)}'$  given by \r{IC_E2}. The second equation is solvable for $H_{(0)}$ in every ball by the Poincar\'e lemma for two-forms. Such a solution is determined up to $\nabla\phi$ for some function $\phi$. Having one such solution $\tilde H_{(0)}$, to satisfy the first equation we need to solve  $\Delta\phi = -\div \tilde H_{(0)}$, which can be done locally. In other words, \r{EH0} can be solved in every ball. Its solution is determined up to $\nabla\psi$ with $\Delta\psi=0$.  On the other hand, adding such a term to $H$ (time-independent) in \eqref{eq1a}--\eqref{eq4a} keeps those equations satisfied, and adds it to the second initial condition in \eqref{IC_EH}. A more general view is that any such  $(0,H_{(0)})$ is a stationary solution of \eqref{eq1a}--\eqref{eq4a} with the same initial conditions, and can be added linearly to any other solution.

Having fixed   $H_{(0)}$, we set 
\begin{equation}   \label{HH0}
 H(\cdot)=H_{(0)}(\cdot)- \mu_0^{-1}\int_{0}^{t}\curl E(\tau,\cdot)\, \d \tau
\end{equation}
which guarantees that \eqref{eq2a} is satisfied. 
 Differentiating \r{eq1a} with respect to $t$, we see that its time derivative vanishes as a consequence of \r{eq:wave1}. Then \r{eq1a} holds up to a constant in $t$ but that constant vanishes because of \r{IC_E2}. This establishes the equivalency of \r{eq1a}--\r{eq4a} and \r{eq:wave1}, including the corresponding initial conditions.

\begin{proposition}  \label{pr1} 
	Let $(E,H)$ be a solution of \eqref{eq1a}--\eqref{eq5} with initial conditions \r{IC_EH} for which $D_{(0)}$, $H_{(0)}$ are divergence free. Then $E$ solves \r{eq:wave1} with initial conditions \r{IC_wave}, where $E_{(0)}'$ can be determined by \r{IC_E2}. 
	
	On the other hand, given $E_{(0)}$, $E_{(0)}'$ so that the left-hand sides of \r{IC_E1}, \r{IC_E2} are divergence free, determine $D_{(0)}'$ by \r{IC_E2}. Let $H_{(0)}$ be a solution of \r{EH0} determined up to the differential of some harmonic function $\phi$. 
	Then, $(E,H)$, with $H$ determined by \eqref{HH0} 
	is a solution of \r{eq:wave1} with initial conditions \r{IC_EH}, with $H$ determined up to  $\nabla \phi$.
\end{proposition}

Using the identity
\begin{equation}\label{eq:curl}
	\curl  \curl  = \n \div  -\Delta ,
\end{equation}  
 and writing  $(\e_0\mu_0)^{-1} = c^2$, equation \eqref{eq:wave1}  reduces to 
 \begin{equation}\label{eq:second_order_metric2}
 	(1+ \chi^{(1)})\p_t^2 E-c^2(\Delta E - \nabla \div E)=-\chi^{(3)}\p_t^2 (|E|^2E).
 \end{equation}
 If $\mu_0$ is the magnetic permeability in vacuum, then $c$ is the speed of light. 
The linear part is a formal elasticity system   with a zero pressure speed. On the other hand, the divergence free condition eliminates pressure waves. %

As mentioned in Section \ref{sec:main_results}, we take $\chi^{(1)}=0$.
Taking $\chi^{(1)}$ variable, of magnitude $\sim 1$, as long as $1+\chi^{(1)}>0$, would make the speed $c(x)$  variable but still $\sim 1$ (i.e., not a large or a small parameter), see also Section \ref{sec_scale}.
Finally,  we rescale the variables to ensure light speed $c=1$, thus arriving to \eqref{eq:second_order_metric3} and \eqref{div0} from \eqref{eq:second_order_metric2} and \eqref{div} respectively. 
According to  \cite[pp.~67--68]{boyd2020nonlinear}, the contribution of the term  $\n \div E$ in \eqref{eq:second_order_metric3} can be considered negligible in major cases of interest. We do not ignore it however, and we show that  it indeed contributes a lower order term in our case.  Its presence  leads to some technical complications, which we resolved.

\section{The weakly nonlinear geometric optics}\label{:the_weakly_geometric_optics} 
\subsection{The general setup}\label{sec_gen} 
Recall that we seek an asymptotic solution to \eqref{eq:second_order_metric3} under the divergence free condition \eqref{div0}, having the form \eqref{A1} in the region of interest.
As we explain   in Section~\ref{sec_scale}, a different scaling of the solution may affect the qualitative effect one observes.
We now explain our choice of the scaling from the point of view of  weakly geometric optics using a simplified model. 
Assume temporarily  that $u\coloneqq  E$ is scalar and real (ignoring the third term in \eqref{eq:second_order_metric3}), and write $\chi$ instead of $\chi^{(3)}$; then $|u|^2u=u^3$. 
Since 
\[
\chi \partial_t^2u^3 =  3\chi u^2\partial_{t}^{2}u +6\chi u (\partial_tu)^2,
\]
we can write a simplified version of \r{eq:second_order_metric3} as
\begin{equation} \label{2}
	(1+3\chi u^2)\partial_t^2 u - c^2\Delta u + 6\chi u (\partial_tu)^2= 0.
\end{equation}
This gives us speed
\begin{equation}\label{cchi}
	c_\chi = (1+3\chi u^2)^{-1/2}c.
\end{equation}
When $\chi u^2$ is of order $\sim 1$, the speed is effectively changed by the solution, so we are in the fully nonlinear regime, where in particular, the geometry of the rays changes (\cite{Metivier-Notes}). The DC Kerr effect does not do that. When $\chi u^2\sim h$, the eikonal equation describing the propagation of singularities (an analog of \eqref{eq:dispersion_rel} below) stays the same as that of the linear wave equation, so the geometry is unchanged, but the nonlinearity  affects the transport (profile) equation for the leading term of the amplitude. 
This is the weakly nonlinear geometric regime by definition. Moreover, if the leading term in the asymptotic expansion of $u$ were constant, we would expect a linear leading transport equation, which we confirm below in the vector valued case. The effective change of the speed by $\sim h$ would cause a retardation $\sim h$ proportional to the wavelength $2\pi h$, which would create a phase shift of a size comparable to the latter over distance $\sim 1$, and this is what we want to model.  The physical $\chi$ is very small in the metric system of units but we can always rescale $u$, which in turn rescales $\chi$ as well (see Section~\ref{sec_scale}), so we do not treat $\chi$ as a ``small'' quantity.  %
The discussion above indeed suggests the scaling $h^{1/2}\underline{E}(x,h)$ for  the strong stationary electric field in a  solution \eqref{A1}. 

We now comment on  our choice of initial conditions \eqref{IC1}.
As mentioned, we want $h^{1/2}\underline{E}(x,h)$ to solve \eqref{div0}, at least in a 
large region $B(0,R_0)$.
This was essentially the content of Theorem \ref{thm_main}\ref{item:a}; we state here a very similar result, whose main addition to the former is a statement regarding the extension of $\underline{E}$ on all of $\R^{3},$ in the case of a constant leading order term. We will prove it together with Theorem \ref{thm_main}\ref{item:a} in Section~\ref{eq:choosing_the_strong_electric_field}.

\begin{lemma}\label{lm:stationary}
Choose a large ball $B(0,R_0)\subset \R^3$ with $R_0\gg R$ and such that $\supp U_{\rm init}\subset B(0,R_0)$ and fix $s>5/2$. There exists a strong field $h^{1/2}\underline{E}(\cdot ,h)\in H^{s-1}_{\mathrm{loc}}(\R^3)\cap L^\infty(\R^3)$ as in \eqref{E0} satisfying \eqref{div0} outside a small neighborhood of $\partial B(0,R_0)$.
It can be taken such that $\u{E}_0=\text{const}.$ on $\R^3 $ and  $\underline{E}_k\in C_0^{\infty}(\R^3)$ for $k\geq 1$, with $\curl \u E_k=0$ on $\R^3$ for all $k\geq 0$.
\end{lemma}
\noindent 	 Lemma~\ref{lm:stationary} implies that the first term in \eqref{IC1} generates a stationary electric displacement field $\underline D$ that is divergence free in a large enough domain.
To make the second term in \eqref{IC1} divergence free as well, first write
\[
\div \left( U_{\rm init}(x) \cos \frac{x\cdot\omega}{h} \right)= \left(\div  U_{\rm init}(x)\right)   \cos \frac{x\cdot\omega}{h} -h^{-1}\omega\cdot   U_{\rm init}(x)\sin  \frac{x\cdot\omega}{h}.
\]
We choose $U_{\rm init}(x)$ so that \eqref{IC0} is satisfied.
If $\omega=e_1 := (1,0,0)$, which we can always assume, this can be achieved by taking 
\begin{equation}   \label{ic1}
  U_{\rm init}(x) = (0,-\partial_{x_3}\rho,\partial_{x_2}\rho ),
\end{equation}  
with some $\rho\in C_0^\infty$. Below, we will also consider $U_{\rm init}$ with the property to   be a non-zero constant field in the ``core'' of the beam. This can be done by taking 
\begin{equation}   \label{ic2}
\rho(x_2,x_3)=-a_2x_3+a_3x_2 \; \text{when $| (x_2,x_3)|\le r_0$} 
\end{equation}
 with some $r_0>0$ (the radius of  ``core'') and $(a_2,a_3)\not=(0,0)$.
Then   $  U_{\rm init}(x) = (0,a_2, a_3 )$  when $| (x_2,x_3)|\le r_0$.

 We now verify that $D_{(0)}$ in \r{IC_E1} is divergence free in the region of interest (with $\chi^{(1)}=0$). The support property of  $U_{\rm init}$ yields
\begin{equation}
	E_{(0)}+  \chi^{(3)} |E_{(0)}|^2 E_{(0)} =   E_{(0)}+ \chi^{(3)} h^{3/2} |\underline{E}|^2 \underline{E} = \big( h^{1/2}\underline{E} + \chi^{(3)} h^{3/2} |\underline{E}|^2\underline{E}\big)  +  2h^{3/2}   U_{\rm init}(x) \cos \frac{x\cdot\omega}{h},
\end{equation}
and both terms are divergence free outside a neighborhood of $\partial B(0,R_0)$. Next, $D_{(0)}'$ in \r{IC_E2} is also divergence free because $E'_{(0)}=0$. 

\begin{remark}
	An alternative way to  choose initial conditions is by taking
	\[
	E|_{t=0}=  h^{1/2}\underline{E}(x,h)+2 h^{3/2}  h\curl \left( \tilde{U}_{\rm init}(x) \sin \frac{x\cdot\omega}{h}\right), \quad \partial_t E|_{t=0}=0.
	\]
Then 	
\[
E|_{t=0}-  h^{1/2}\underline{E}(x,h) =  h^{3/2} \left[(\omega\times   \tilde{U}_{\rm init}(x)) \cos \frac{x\cdot\omega}{h} +  h\left(\curl  \tilde{U}_{\rm init}(x) \right)\sin \frac{x\cdot\omega}{h}\right]
\]
and $ U_{\rm init} \coloneqq  \omega\times   \tilde{U}_{\rm init}$ is normal to $\omega$ but not automatically divergence free as above (but it does not need to be), unless we chose it properly. We can make $\tilde{U}_{\rm init}$ constant in the core of the beam as well. 
\end{remark}

\subsection{Initialization of the problem, the linear solution} \label{sub:initialization_of_the_problem_the_linear_solution}
For $0\leq t\ll 1$, \eqref{eq:second_order_metric3} and \eqref{IC1} have a solution  in the linear regime, and an expansion for it in $h$ can be constructed using geometric optics for the scalar wave equation in $\R^{3}$.
Indeed, if we produce solutions of the equation
\begin{equation}\label{eq:simplified}
	\partial_t^2 E-\Delta E=0,
\end{equation}
with $E$ vector valued and with initial conditions satisfying $\div E\big|_{t=0}=\div \partial_t E\big|_{t=0}$, then  $\div E=0$ stays true for all $t$.
In the previous section we arranged the initial conditions in such a way that the divergence free condition holds in a neighborhood of $\supp U_{\rm init}$.
This implies that for such initial data, a solution of \eqref{eq:simplified} (which will have finite speed of propagation) will also solve \eqref{eq:second_order_metric3} for small enough $t$, and it will be divergence free near $\supp U_{\rm init}$.
By our uniqueness result (see Proposition~\ref{prop_Gues}), this solution will be the only solution to \eqref{eq:second_order_metric3}, \eqref{IC1} in a suitable class.  
Near $\supp U_{\rm init}$, we have
\begin{equation} 
\begin{aligned}\label{A1'}
  		 E = h^{1/2}\underline{E}(x,h)  + h^{3/2} U_{\rm in }(t,x,\phi_{\rm in }/h,h)+h^{3/2} U_{\rm out }(t,x,\phi_{\rm out }/h,h),\quad 	\phi_{\rm in/out}= \mp t+x\cdot \omega
  \end{aligned}  
\end{equation}
and, for $t$ small,
\begin{equation}\label{Uout'}
	\begin{aligned} 
 U_{\rm in}(t,x,\theta,h) &=   e^{\i \theta} a_{{\rm in}, +}(t,x,h)+ e^{-\i \theta} a_{{\rm in}, -}(t,x,h),\\
U_{\rm out}(t,x,\theta,h) &= e^{\i \theta} a_{{\rm out}, +}(t,x,h)+  e^{-\i \theta } a_{{\rm out}, -}(t,x,h).%\label{Uout'}%
\end{aligned}	
\end{equation}
The amplitudes have asymptotic expansions 
\begin{equation}   \label{ain}
a_{\rm in/out,\pm} \sim \sum_{k=0}^\infty h^k a_{\rm in/out,\pm}^{(k)} (t,x).
\end{equation}
Substituting into \eqref{eq:simplified} and matching powers of $h$, the transport equations for the leading order amplitudes are found to be
\begin{equation}   \label{eq:tr0}
  (\p_t +\omega\cdot\n ) a_{\rm in,\pm}^{(0)} = 0, \quad  (\p_t -\omega\cdot\n ) a_{\rm out,\pm}^{(0)} = 0
\end{equation}
with initial conditions satisfying
\[
a_{\rm in,+}^{(0)} + a_{\rm out,+}^{(0)}=U_{\rm init}, \quad 
a_{\rm in,-}^{(0)} + a_{\rm out,-}^{(0)}=U_{\rm init}   \quad\text{for $t=0$} .
\]
The second initial condition in \r{IC1} implies
\[
-a_{\rm in,+}^{(0)} + a_{\rm out,+}^{(0)}=0, \quad 
a_{\rm in,-}^{(0)} - a_{\rm out,-}^{(0)}=0   \quad\text{for $t=0$} .
\]
Therefore,
\[
a_{\rm in,+}^{(0)} = a_{\rm out,+}^{(0)}=a_{\rm in,-}^{(0)} = a_{\rm out,-}^{(0)}= \frac12 U_{\rm init}  \quad\text{for $t=0$} .
\]
Solving the transport equations, we get 
\[
a_{\rm in,+}^{(0)} = a_{\rm in,-}^{(0)}=\frac12 U_{\rm init}(x-t\omega),\quad 
a_{\rm out,+}^{(0)} = a_{\rm out,-}^{(0)}= \frac12 U_{\rm init}(x+t\omega). \quad 
\]
Thus the leading term of the beam is 
\begin{equation}   \label{U0}
 U_{\rm init}(x-t\omega)\cos\Big(\frac{x\cdot\omega-t}{h}\Big) + U_{\rm init}(x+t\omega)\cos\Big(\frac{x\cdot\omega+t}{h}\Big),
\end{equation}
representing two ``waves'' (not solutions of the wave equation by themselves) moving in opposite directions. 
The lower order terms of the amplitudes can be computed consecutively in a similar way, solving non-homogeneous versions of the transport equations  \eqref{eq:tr0} (that is, with the same vector fields), with sources depending on  terms we already computed and their derivatives. We can construct as many terms in the expansion as we want, say up to $h^N$ with a remainder $O(h^{N+1})$. 
Note that each of $U_{\rm in/out}$ propagates with finite speed, and is smooth if $U_{\rm init}$ is.

\subsection{The asymptotic solution: the leading term \texorpdfstring{$U_0$}{U zero}} 
  \label{sub:the_asymptotic_solution_the_leading_term}
Our goal is to construct  an asymptotic solution of \eqref{eq:second_order_metric3}, \eqref{IC1} for $t\in [0,T]$ with a fixed $T>0$. 
Since $U_{\rm out}$ propagates away from $B(0,R)$  in a linear regime,
the terms in its expansion do not contribute to this solution there.
We denote 
\begin{equation}
	 {E_{\rm in}} = h^{1/2}\underline{E}(x,h)  + h^{3/2} U_{\rm in }(t,x,\phi_{\rm in }/h,h)
\end{equation}
 and use the ansatz
\begin{equation}\label{ansatz}
	  E_{\rm in}
	\sim h^{1/2}\u{E}_0+h^{3/2}\big(\u{E}_1(x)+U_0(t,x,\phi/h)\big)+h^{5/2}\big(\u{E}_2(x)+U_1(t,x,\phi/h)\big)+\dots,
\end{equation}
where $\phi=\phi_{\rm in }=-t+x\cdot \omega$.
A more accurate notation would be $U_{\rm{ in},k} $ instead of $U_k$, but we will keep the latter to improve readability.
If $K \geq 0$, we say that a formal $h$-dependent expansion for $E$ as in \eqref{A1'} solves \eqref{eq:second_order_metric3} up to order $O(h^{K-1/2})$ if 
\begin{equation}\label{eq:solve_to_order}
		\p_t^2 E- \Delta E + \nabla \div E+\chi^{(3)}\p_t^2 (|E|^2E)\sim \sum_{k\geq K}h^{k+1/2}F_k(t,x,\phi_{\rm in}/h,\phi_{\rm out}/h),
\end{equation}
where $F_k$ are smooth and $2\pi$-periodic in the last two variables, and
 similarly for $ E_{\rm in}$ and $h^{3/2}U_{\rm out}$.  

The next lemma, proved in Section \ref{sec:the_profile_equations_proofs_of_lemmas_ref_lm_eikonal_transport}, confirms the predictions in Section ~\ref{sec_gen}, namely that the  nonlinear term $\chi^{(3)}\partial_t (|E|^2E)$ does not affect the dispersion relation, but it does affect the transport equation for the  first profile $U_0$.
We write $\Pi_\omega$ and $\Pi_\omega^{\perp}$ for the orthogonal projections onto the span of $\omega$ and its orthogonal complement respectively.

\begin{lemma}\label{lm:eikonal_transport}
Suppose that  $ E_{\rm in}$ in \eqref{ansatz} satisfies \eqref{eq:second_order_metric3} to order $O(h^{1/2})$.
Then the phase function $ \phi =\omega\cdot x - t$ is characteristic, that is, it satisfies 
the dispersion relation\footnote{Recall that $\nabla$ denotes gradient only in $x$.} 
\begin{equation}\label{eq:dispersion_rel}
	\det \left( ((\partial_t \phi)^2-|\nabla \phi|^2 )\Id_3+\nabla \phi\otimes \nabla \phi\right)=0.
\end{equation} 
With this  $\phi$, and if $\u{E}_0$ in \eqref{E0} is constant with $\u{E}_0\cdot \omega=0$, the first profile $U_0$ satisfies the equations 
\begin{align}
		 \partial_\theta \Pi_\omega U_0=0, \label{tE0}\\
	\mathcal{L}( \u{E}_0)\Pi_\omega^\perp \p_\theta U_0=0,
\label{tEp}
\end{align}
where $\mathcal{L}$ is the vector valued transport operator
\begin{equation}\label{eq:l_op}
	\mathcal{L}(\u{E}_0)V=-2(\partial_t + \omega\cdot\nabla) V +\chi^{(3)} |\u{E}_0|^2(\partial_\theta V)+2\chi^{(3)} (\u{E}_0\cdot \partial_\theta V) \u{E}_0.
\end{equation}
If $\phi$ is as above and \eqref{tE0} is satisfied, $E_{\rm in}$ solves \eqref{eq:second_order_metric3} to order $O(h^{-1/2})$.  
\end{lemma}

Since $U_0$ is assumed to be periodic in $\theta$, it has a Fourier series expansion. 
At the expense of adding a $\theta$-independent term to $U_0$, we can cancel one power of $\p_\theta$ in \eqref{tE0}-\eqref{tEp}. 
The initial condition  $U_0=U_{\rm init}(x)\cos\theta$ at $t=0$ in \r{U0} is of single phase, with $\theta=\phi/h$  (more precisely, with $\phi/h$ and $-\phi/h$ in complex notation), so 
since \eqref{tE0}-\eqref{tEp} are  linear, no higher order harmonics are generated in the principal term $U_0$.
Thus its Fourier series  reduces to  
\begin{equation}   \label{U00}
U_0 =  A_0(t,x)\cos\theta + B_0(t,x)\sin\theta + C_0(t,x),
\end{equation}
where from \eqref{tE0} it follows that $A_0,B_0$ must take values in $\omega^{\perp}$ and $C_0$ takes values in $\R^3$ (unconstrained, though we will see in Section \ref{sub:the_asymptotic_solution_full_expansion} that it vanishes). 
Introduce the characteristic variables
\begin{equation}\label{eq:char_coordinates}
	(s,y) = (t, x-t\omega); \quad \text{then $(t,x) = (s, y+s\omega )$} \text{ and }\partial_t + \omega\cdot\nabla_x = \partial_s.
\end{equation}
Passing to the characteristic coordinates,  
\eqref{tEp} integrated in $\theta$
yields
\begin{align*}
-2\partial_s (A_0\cos\theta +B_0\sin\theta)& +  \chi^{(3)} |\u{E}_0|^2 (-A_0 \sin\theta + B_0\cos\theta)\\
 &-  2\chi^{(3)}  \sin\theta (\u{E}_0\cdot A_0) \u{E}_0 +2\chi^{(3)} \cos\theta (\u{E}_0\cdot B_0) \u{E}_0 =0. 
\end{align*}
Note that this ODE carries no information about $C_0$. 
It is equivalent to the $6\times 6$ system
\be{AB}
\frac{\d}{\d s}\begin{pmatrix*} A_0\\B_0 \end{pmatrix*} +\frac12 \chi^{(3)} |\u{E}_0|^2 \begin{pmatrix*} -B_0\\A_0 \end{pmatrix*} +\chi^{(3)} \begin{pmatrix*}- (\u{E}_0\cdot B_0)\u{E}_0\\ (\u{E}_0\cdot A_0)\u{E}_0 \end{pmatrix*} = 0.
\ee
Note that all quantities above are restricted to a characteristic line, and are functions of $s$ and of the other parameters determining the line. 
We recast \eqref{AB} as
\begin{equation}   \label{eq:tr'a}
	\left[ \frac{\d}{\d s} + \frac12 \chi^{(3)}  \begin{pmatrix*} 0&-|\u{E}_0|^2\Id_3 -2\u{E}_0\otimes \u{E}_0\\ |\u{E}_0|^2\Id_3 +2 \u{E}_0\otimes \u{E}_0&0 \end{pmatrix*} 
	\right] \begin{pmatrix*} A_0\\B_0 \end{pmatrix*} = 0.
\end{equation}

Next, we assume without loss of generality that $\u{E}_0=|\u{E}_0|e_3$. Equation \eqref{eq:tr'a} becomes
% \begin{equation}   \label{eq:tr}
% \left[ \frac{\d}{\d s} + \frac12 \chi^{(3)} |E_0|^2 \begin{pmatrix*} 0&-\Id_3\\ \Id_3&0 \end{pmatrix*} 
% +       \chi^{(3)}  |\u{E}_0|^2  \begin{pmatrix*} 0&-\mathrm{diag}(0,0,1)\\ \mathrm{diag}(0,0,1)&0 \end{pmatrix*} 
% \right] \begin{pmatrix*} A_0\\B_0 \end{pmatrix*} = 0,
% \end{equation}
% or
\[
\left[ \frac{\d}{\d s} + \frac12 \chi^{(3)} |\u{E}_0|^2  \begin{pmatrix*} 0&-\mathrm{diag}(1,1,3)\\ \mathrm{diag}(1,1,3)&0 \end{pmatrix*} 
\right] \begin{pmatrix*} A_0\\B_0 \end{pmatrix*} = 0.
\]
Set $\mathcal{A}=A_0+\i B_0$. 
Temporarily adopting the notation   $\tilde \chi^{(3)}(s) = \chi^{(3)}(y+s\omega)$, we get
\begin{equation}
\frac{\d}{\d s}\mathcal{A} =  \frac12 \tilde \chi^{(3)}(s) |\u{E}_0|^2 \begin{pmatrix*}  -\i &0&0\\ 0& -\i &0\\0&0& -3 \i  \end{pmatrix*} \mathcal{A},\label{eq:complex_sys}	
\end{equation}
which  is a diagonal system. 
Recall \eqref{tau} and write
\begin{equation}\label{eq:tau_tilde}
	\tilde \tau(s) \coloneqq  \tau(y+s\omega)= \frac12|\u{E}_0|^2 \int_{-\infty}^s \tilde \chi^{(3)}(\sigma)\,\d\sigma,
\end{equation}
The solution of \eqref{eq:complex_sys} is now seen to be $$\mathcal{A}= (c_1e^{-\i \tilde \tau}, c_2e^{-\i\tilde  \tau}, c_3e^{-\i3 \tilde  \tau}).$$ So we get components with fundamental periods $2\pi$, $2\pi$, $2\pi/3$. In particular, $\u{E}_0$ creates a phase shift of the third component different that is from the other two.

We now take the initial condition \r{U0} into account. In characteristic coordinates, $\phi=y\cdot\omega$. %
The initial condition for $\mathcal{A}$ is $\mathcal{A}= U_{\rm init}(y\cdot\omega)$ for $0\leq t\ll 1$. Note that we get the initial condition $C_0=0$ when $t=0$, which we will use later. Since we chose the direction of propagation to be perpendicular to $\u{E}_0$, we have $\omega_3=0$, and without loss of generality we can assume $\omega=(1,0,0)$ as in Theorem \ref{thm_main}. Since $U_{\rm init}$ is orthogonal to $\omega$ by \eqref{IC0}, it is of the form $U_{\rm init} = (0, U_{{\rm init},2}, U_{{\rm init},3})$ for $0\leq t\ll 1$. 
Then $\mathcal{A}=  \big(0,U_{{\rm init},2}(y)  e^{-\i \tilde \tau}, U_{{\rm init},3}(y) e^{-\i 3\tilde \tau}\big)$,
where $U_{\rm init}$ is written in the coordinates $y=(x_1-t,x_2, x_3)$, $s=t$. That gives us a leading term of the solution 
% \begin{equation}
	\begin{align}
	 E_{\rm in}&\sim h^\frac12 |\u{E}_0|e_3+ h^\frac32C_0(t,x) +h^\frac32  \Big(0, U_{{\rm init},2}\cos\tilde \tau \cos(\phi/h) - U_{{\rm init},2}\sin\tilde \tau\sin(\phi/h),\\*
	& \qquad \qquad \qquad \qquad  \qquad U_{{\rm init},3} \cos(3\tilde \tau) \cos(\phi/h) -U_{{\rm init},3} \sin(3\tilde \tau)\sin(\phi/h) \Big)+O(h^{5/2})\label{Ein}\\
	&=  h^\frac12 |\u{E}_0|e_3+h^\frac32C_0(t,x)+ h^\frac32  \big(0, U_{{\rm init},2}\cos(\phi/h+\tilde \tau), U_{{\rm init},3}\cos(\phi/h+3\tilde \tau)\big)+O(h^{5/2}),
\end{align}
% \end{equation}
where $U_{\rm init} =U_{\rm init}(y) $, $\tilde \tau=\tau(y+t\omega)$, and  $\phi= y\cdot\omega = -t+x\cdot\omega$ with $\omega=e_1$.
This yields \eqref{exp1}, and \eqref{eq:exp_no_nonlinearity} upon setting $\u E_0=0$, except for the fact that there we have $C_0=0$ (this is shown in Section~\ref{sub:the_asymptotic_solution_full_expansion}) and for the statement there regarding the error (explained in Section \ref{sec:exact_solutions}).  

If we assume $x=y+s\omega$ with $s\gg1$, so that the ray leaves $\supp\chi^{(3)}$, then $\tau=\text{const.}$,  and the third term in \eqref{Ein} stays constant along that ray. In the standard $(t,x)$ coordinates, $\phi=-t+ x\cdot\omega$, and under the same assumption about $\tau$ and for a fixed $x\cdot\omega$, the second term of the electric field $E$ rotates on an ellipse. This is a circular polarization when $U_{{\rm init},2}=U_{{\rm init},3}\not=0$ and $2\tau=(2k+1)\pi/2$, $k$ an integer. It is elliptical in general. 
When $U_{{\rm init},2}=U_{{\rm init},3}=1$, for example,  the half-lengths of the axes are $\sqrt2 \cos\tau$, $\sqrt2 \sin\tau$, respectively. 

We demonstrate this effect in  Figure~\ref{pic4_pr_2015}. The plot represents the change of the polarization along a light ray, i.e., for $s\mapsto (t,x)=(s,y_0+se_1)$ with $y_0$ fixed, and $0<h$ not too small to make the diagrams less cramped. There, $s$, which is also $x_1$ shifted, is along the horizontal axis.  
This is the basis for solving the inverse problem in Theorem~\ref{thm_inverse}. Measuring the change of the polarization, encoded by the phase shift $\tau$, we recover the X-ray transform of $\chi^{(3)}$, and then $\chi^{(3)}$ itself.  
\begin{figure}[!ht] %
  \centering 
\hfill \includegraphics[scale=0.5]{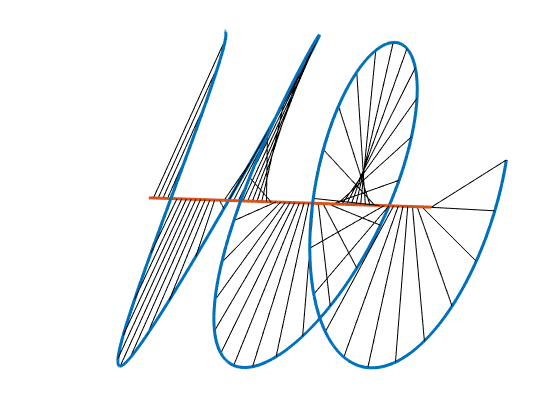}\hfill 
\includegraphics[scale=0.45]{DC_Kerr_pic2.png}\hfill \ 
\caption{\small The polarization changes from linear to ellipsoidal along the way. The two axes are preserved. Left: The polarization starts as a linear one but the nonlinear interaction converts it into elliptical. Right: Frontal view.}
  \label{pic4_pr_2015}
\end{figure} 

\subsection{Interpretation and Kerr cells} \label{sec_cell}
The interpretation of the nonlinear effect in the leading term in \r{exp1} is as follows. The ``strong'' electric field $h^{1/2}\underline E$ changes the effective matrix-valued index of refraction making it birefringent -- that is, creating speeds of light depending on the polarization: different for light  polarized in the direction of $\underline E$ compared to that for light perpendicular to it. The change in speeds between the two polarization directions is proportional to the wavelength, to the square of the amplitude of the electric field, and to a constant depending on the material.  In our analysis, that change is $2 h\tau $, $\tau\sim |\u{E}_0|^2$, which we view as a  phase shift in \r{exp1}. The nonlinearity always produces a phase shift. When $U_{{\rm init},3}=0$ but $U_{{\rm init},2}\not=0$, the light, a priori linearly  polarized in the direction $U_{\rm init} = (0,U_{{\rm init},2}, U_{{\rm init},3}) =(0, U_{{\rm init},2},0)  $, remains linearly polarized  in the same direction but with a phase shift $-\tau$,
increasing along the way. 
Similarly, when $U_{{\rm init},2}=0$ but $U_{{\rm init},3}\not=0$, the light, a priori  linearly polarized in the direction $U_{\rm init} = (0,U_{{\rm init},2}, U_{{\rm init},3}) =(0,0, U_{{\rm init},3})  $, remains linearly polarized as well but the phase shift is $-3\tau$.
By \r{tau}, this is consistent with the heuristic argument in Section~\ref{sec_gen}, where we estimated a reduction of the speed by $\frac32 h\chi^{(3)} |\u{E}_0|^2+O(h^2)$ in \r{cchi}, since in our case $E~\sim h^{1/2}\u{E}_0$. 
When neither $U_{{\rm init},2}$ nor $U_{{\rm init},3}$ vanish, the polarization becomes elliptical. 

Kerr cells, or Kerr cell shutters, are based on the DC Kerr effect, see \cite{J_C_Filippini_1975}. Two polarizing filters are placed at angle $\pi/2$ to each other, one at the entrance of a beam and one at the exit from a cell filled with a material with relatively large nonlinear coefficient $\chi^{(3)}$, for example nitrobenzene. In Figure~\ref{pic4_pr_2015} and in Figure~\ref{pic2} for example, the first filter polarizes the beam linearly as shown there at the left-hand side.  If $\underline E=0$, i.e., no stationary electric field is applied (not shown in the figure), the polarization remains constant, as it follows from our analysis. Then the second filter blocks the light. Applying a ``strong'' $\underline E$ causes the polarization to rotate, and if at the exit the total rotation is not $k\pi $, $k\in \mathbb{N}$, at least some light would be transmitted. Switching $\underline E$ on and off can happen very quickly, giving us a shutter speed far exceeding that  of a mechanical one. 

\begin{figure}[h!] %
	\centering 
	\includegraphics[scale=0.5]{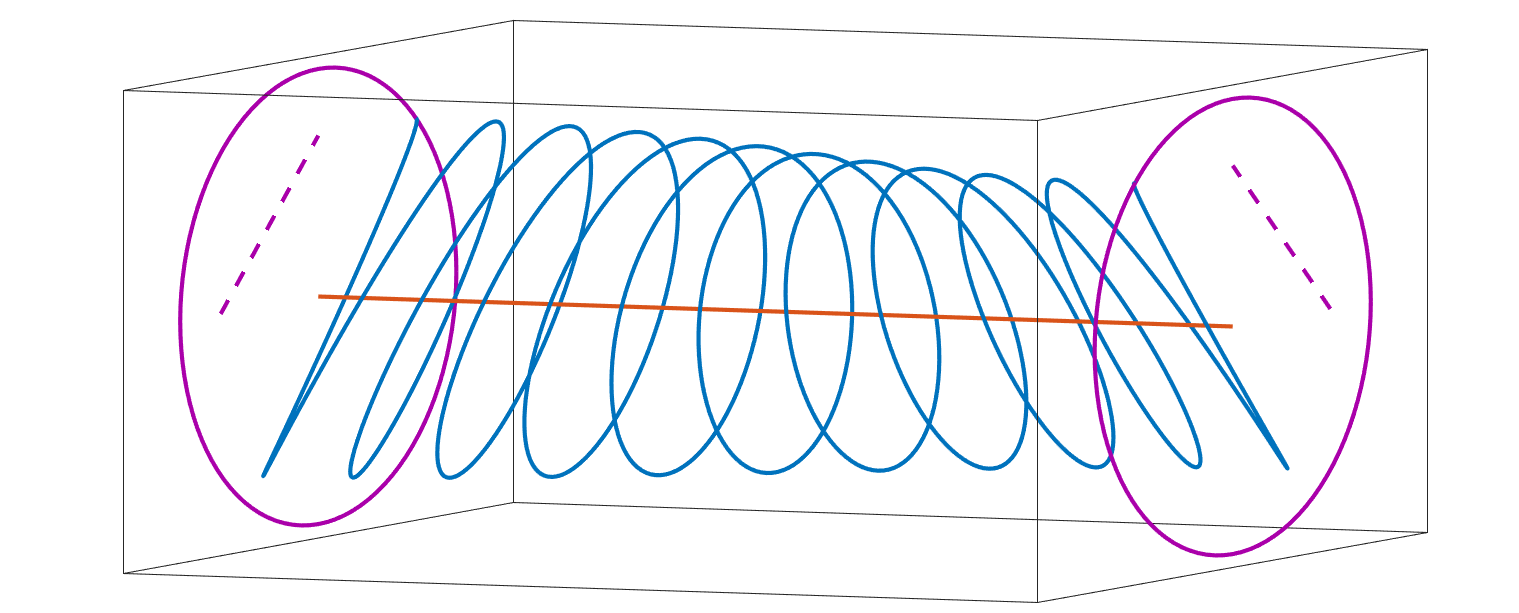}
	\caption{\small Polarization change with a constant $\chi^{(3)}$. The parameters are chosen so that the total rotation is by $\pi/2$. The round polarizing filters have orientations along the dotted lines, and the one on the right transmits all the light.}
	\label{pic2}
\end{figure} 

We obtain maximal transmission when the total rotation  at exit is by $\pi/2 +k\pi$, see Figure~\ref{pic2}. Assuming for the moment that $\chi^{(3)}$ is constant in the region of interest (although then $\chi^{(3)}$, extended globally as zero, is discontinuous, and we do not study that case), by our analysis the phase shift of one component relative to the other is $3\tau-\tau=2\tau$ by absolute value. Thus the maximal effect takes place when $\tau = \pi/2+k\pi$.  On the other hand, $\tau=\frac12 |\u{E}_0|^2 \chi^{(3)}d$, where $d$ is the distance between the polarizers. Therefore, the equation
\[
\frac12 |\u{E}_0|^2 \chi^{(3)}d = \frac\pi 2
\]
can be used to determine $\chi^{(3)}$ when $\u{E}_0$ is given, by varying $\u{E}_0$ from $0$ to the value giving us the highest intensity for the first time (thus eliminating $k\pi$). This method was used in \cite{al2014determination}, 
for example, to measure the Kerr constant, which in turn determines $\chi^{(3)}$ for olive oil. On the other hand, if we know $\chi^{(3)}$, we can find the optimal $\u{E}_0$ maximizing the output. 

To formalize these arguments, assume that $U_{\rm init}$ is as in \eqref{ic1}, \eqref{ic2}, making $U_{\rm init} = (0,a_2,a_3)$ in the ``core'' $x_2^2+ x_3^2\le \rho^2$ of the beam along the axis $\omega=e_1$. The second polarizing filter projects the high frequency part of the beam to $(0,-a_3,a_2)/\sqrt{a_2^2+ a_3^2}$. 
Then the electric energy density $\frac{\e_0}2 |E|^2 $ of the high frequency part of \r{exp1} in the core is proportional to 
\begin{equation}   \label{a1a2}
	\frac{a_2^2a_3^2}{a_2^2+a_3^2} \sin^2 (\tau) \sin^2( \phi/h +2\tau).
\end{equation}
Indeed, the effect of the polarizing filter is to project the term in the second line in \eqref{exp1} to $(0,-a_3,a_2)/\sqrt{a_2^2+ a_3^2}$, so what is left propagating is
\begin{align}
	-\frac{a_2a_3}{{a_2^2+a_3^2}} &\big(\cos(\phi/h +\tau) - \cos(\phi/h +3\tau)\big)(0,-a_3,a_2)\\
	&= 2\frac{a_2a_3}{{a_2^2+a_3^2}} \sin(\phi/h +2\tau)  \sin(\tau) (0,-a_3,a_2).
\end{align}
Its energy is proportional to its norm squared, thus giving us \eqref{a1a2}. 
The magnitude of the amplitude $2 \frac{|a_2a_3|}{\sqrt{a_2^2+a_3^2}} |\sin(\tau)| $ of the oscillations attains its maximum  with respect to $\tau$ when $\tau=\pi/2+k\pi$, provided $a_2a_3\not=0$, as stated above. 
In terms of $(a_3,a_3)$   with a fixed length, the maximum  is achieved when $\tau\not=k\pi$ and $a_2=a_3$, i.e., when the angle with $e_3$ is $\pi/4$. 
{}
\subsection{The asymptotic solution: full expansion} \label{sub:the_asymptotic_solution_full_expansion}
To prove the existence of exact solutions in Section \ref{sec:exact_solutions}, we need to prove that it is possible to construct an approximate solution to \eqref{eq:second_order_metric3} to high order.
As before, we are looking for the incoming part $ E_{\rm in}$ of a solution, as in \r{ansatz},
with all  of the $U_j$ $2\pi$-periodic in $\theta$, extending $U_{\rm in}$  from the linear region ($0\leq t\ll1$) to the nonlinear one.
Lemmas \ref{lm:U1} and \ref{lm:C0} below will be proved in Section \ref{sec:the_profile_equations_proofs_of_lemmas_ref_lm_eikonal_transport}.

\begin{lemma}\label{lm:U1}
Suppose that \eqref{ansatz} satisfies \eqref{eq:second_order_metric3} to order $O(h^{K+3/2})$, where $K\geq 0$.
Assume also that $U_0$ satisfies \eqref{tE0}--\eqref{tEp}, $\phi=-t+\omega$, and the first $K+2$ terms in the expansion of $\underline{E}$ have been constructed as in Lemma \ref{lm:stationary}, with $\u{E}_0=\text{const.}$ and $\u{E}_0\cdot \omega=0$.
Then for $0\leq k\leq K$ the profiles $U_{{k+1}}$ satisfy equations of the form
\begin{align}
	&	\Pi_\omega \partial_\theta U_{k+1}=\int \Pi_\omega F_k\big(\chi^{(3)},\u{E}_j,\partial^\alpha_{(t,x,\theta)}U_\ell:j,\ell=0,\dots, k,|\alpha|\leq 2\big)\d \theta,\label{eq:U_1_first}\\
	&\begin{aligned}
			&\mathcal{L}(\u{E}_0)\Pi_\omega^{\perp}\p_\theta U_{k+1}=\Pi_\omega^{\perp}\nabla( \partial_\theta U_{k+1}(t,x,\theta)\cdot \omega)\\*
			&\qquad +\Pi_\omega^{\perp} G_k(\chi^{(3)}, \underline{E}_j,\partial^\alpha_{(t,x,\theta)}U_\ell:j=0,\dots k+1,\ell=0,\dots ,k ,|\alpha|\leq 2),	
			\end{aligned}
	\label{eq:U_1_second}
\end{align}
where $\mathcal{L}$ is defined in \eqref{eq:l_op} and each entry of $F_k$, $G_k$ is a polynomial in the entries of their arguments;
  the antiderivative  in \eqref{eq:U_1_first} is chosen such that the zeroth Fourier mode of the right hand side vanishes. 

If there exist smooth  profiles $U_{k+1}$, $-1\leq k\leq K$, such that \eqref{tE0}--\eqref{tEp} and \eqref{eq:U_1_first}--\eqref{eq:U_1_second} hold on a time interval $[0,T]$, then $E$ solves \eqref{eq:second_order_metric3} up to order $O(h^{K+1/2})$.
If in addition those profiles (together with their outgoing counterparts) satisfy the initial conditions determined by \eqref{IC1} as explained in Section \ref{sub:initialization_of_the_problem_the_linear_solution}, then for all such $k$, $\Pi_\omega \partial _\theta U_{k+1}$ (resp. $\Pi_\omega^{\perp}\partial _\theta U_{k+1}$)  contains no complex Fourier modes outside the range $[-k-1,k+1]$ (resp. $[-k-2,k+2]$), and $\partial_\theta U_{k+1}\in C_0^\infty([0,T]\times \R^3\times S^1)$.
\end{lemma}

\begin{remark}
	The appearance of two equations satisfied by $\partial_\theta U_{k+1}$ is typical of the WKB method for constructing approximate solutions for symmetric hyperbolic systems of PDE, see \cite{Metivier-Notes}.
\end{remark}

Note that \eqref{eq:U_1_first} and \eqref{eq:U_1_second} are equations for $\partial_\theta U_{k+1}$, however their right hand sides contain $U_\ell$ and its derivatives for $\ell\in \{0,\dots,k\}$. So in general those equations cannot be solved successively to determine $\partial_\theta U_{k+1}$, unless the zero Fourier modes of the $U_\ell$'s have also been determined for $\ell\leq k$.
We explain now how to solve \eqref{eq:U_1_first}--\eqref{eq:U_1_second}, assuming for the moment that this is done;
the zero modes are the subject of Lemma \ref{lm:C0} below.
First, \eqref{eq:U_1_first} clearly determines $\Pi_\omega U_{k+1}$ up to its zero mode.
For $\Pi_\omega^{\perp}U_{k+1} $, let $ U_{k+1}=\sum _{\ell=-k-2}^{k+2}U_{k+1}^{\{\ell\}} e^{i\ell \theta}$, with $U_{k+1}^{\{\ell\}} = \overline{U}_{k+1}^{\{\ell\}}$ and valued in $\C^3$.
Having determined $\partial_{\theta}U_{k+1}\cdot \omega$ and the previous profiles, the right hand side of  \eqref{eq:U_1_second} is a known source and from the proof of Lemma \ref{lm:U1} follows that it contains no Fourier modes outside the range $[-k-2,k+2]$. Equating modes corresponding to $0<|\ell|\leq k+2$ we obtain
\begin{align}\label{Fourier_exp}
			&\mathcal{L}(\u{E}_0)\Big(\sum _{\ell=-k-2}^{k+2}i\ell \Pi_\omega^{\perp }U_{k+1}^{\{\ell\}} e^{i\ell \theta}\Big)= \sum _{0<|\ell|\leq k+2}f_{k+1}^{\{\ell\}} e^{i\ell \theta},
\end{align}
where the right hand side is the Fourier expansion of the source with respect to $\theta$.
At $t=0$, $ \Pi_\omega^{\perp }U_{k+1}^{\{\ell\}}=0$ for all $k\geq 0$.
Using  \eqref{eq:l_op}, and that if $\u E_0=|\u E_0|e_3$ we have
\begin{equation}\label{eq:diagonalization}
 	|\u{E}_0|^2\Id_3 + 2\u{E}_0\otimes \u{E}_0 =|\u{E}_0|^2\diag (1,1,3),
 \end{equation}\eqref{Fourier_exp} becomes in characteristic coordinates \eqref{eq:char_coordinates}
\begin{align}
			\sum _{0<|\ell|\leq k+2}i\ell\big( -2\Id_3\partial_s +i\ell\tilde\chi^{(3)}(s)|\u{E}_0|^2\diag(1,1,3)\big) \Pi_\omega^{\perp }U_{k+1}^{\{\ell\}} e^{i\ell \theta}= \sum _{0<|\ell|\leq k+2}f_{k+1}^{\{\ell\}} e^{i\ell \theta},
\end{align}
or equivalently, using the notation \eqref{eq:tau_tilde},
\begin{align}
			&\big( \Id_3\partial_s -\frac12i\ell\tilde\chi^{(3)}(s)|\u{E}_0|^2\diag(1,1,3)\big) \Pi_\omega^{\perp }U_{k+1}^{\{\ell\}} =\frac{i}{2\ell}f_{k+1}^{\{\ell\}} \\ 
			&\iff \partial_s\left(\diag(e^{-i\ell \tilde{\tau}},e^{-i\ell \tilde{\tau}},e^{-3i\ell \tilde{\tau}}) \Pi_\omega^{\perp }U_{k+1}^{\{\ell\}}\right)
			 =\frac{i}{2\ell}\diag(e^{-i\ell \tilde{\tau}},e^{-i\ell \tilde{\tau}},e^{-3i\ell \tilde{\tau}}) f_{k+1}^{\{\ell\}} \\ 
			 &\iff 
			  \Pi_\omega^{\perp }U_{k+1}^{\{\ell\}} = \frac{i}{2\ell}\int_{0}^{s}\diag(e^{i\ell (\tilde{\tau}(s)-\tilde{\tau}(\sigma))},e^{i\ell (\tilde{\tau}(s)-\tilde{\tau}(\sigma))},e^{3i\ell (\tilde{\tau}(s)-\tilde{\tau}(\sigma))}) f_{k+1}^{\{\ell\}} \d \sigma
\end{align}
for all $0<|\ell| \leq k+2$, which solves \eqref{eq:U_1_second} at the non-zero modes.

Note that \eqref{Fourier_exp} does not contain the zero mode of the right hand side of \eqref{eq:U_1_second}. The requirement that the  zero mode of \eqref{eq:U_1_second} must vanish is what gives us an equation for the zero mode of $U_k$:

\begin{lemma}\label{lm:C0}Let the notations and assumptions be as in the first part of Lemma \ref{lm:U1} and denote by $A^{\{0\}}$ the zeroth Fourier mode of an expression $A$.
Then $C_k\coloneqq U_k^{\{0\}}$ solves the linear wave-type equation
\begin{equation}
	\begin{aligned}\label{eq:0-mode}
			\qquad  \partial_t^2&C_k-\Delta C_k+\nabla \div C_k  =
			% (\chi^{(3)}(\u{E}_0\cdot \partial_\theta U_k) (\partial_\theta U_0))^{\{0\}}+(2\chi^{(3)}(\u{E}_0\cdot\partial_\theta U_0) \partial_\theta U_k)^{\{0\}} \\
			%
			%
			%
			\chi^{(3)}H_k(\u{E}_j,\partial^\alpha_{t}U_\ell:j=0,\dots, k-1,\ell=0,\dots ,k-1 ,|\alpha|\leq 2)^{\{0\}},	
			\end{aligned}
\end{equation}
where each entry of $H_k$ is a polynomial in the entries of its arguments, and each term contains at least one factor of $ \partial^\alpha_{t}U_\ell$ for some $\alpha$ and $\ell$ as above.
 For $k=0,1$ the right hand side vanishes.
\end{lemma}

We check that \eqref{eq:0-mode}  is solvable for $C_k$ with vanishing initial conditions, provided all the quantities in the source have been determined for $t\in [0,T]$ (namely, $E_j$ for $0\leq j\leq k$, $U_\ell$ for $0\leq \ell\leq k-1$): 
indeed, we can write it as 
\begin{equation}\label{eq:Ck_eq}
	\partial_t^2C_k-\Delta C_k+\nabla \div C_k  =F_k,
\end{equation}
where $F_k$ is known, with $\supp F_{k}\subset  [0,T]\times \supp \chi^{(3)}$.
Moreover, $F_k(t,\cdot )=0$ for $0\leq t\ll 1$, because $\supp \chi^{(3)}\cap \partial^\alpha_{(t,\theta)}U_\ell=0$ for such $t$.
Applying $\div$ to both sides results in $\partial_t^2\div C_k=\div F_k$, from which  $\div C_k=\int _0^{t}\int _{0}^{s}\div F_k(u,x)\d u\d t$, since $C_k\big|_{t=0}=\partial_t C_k\big|_{t=0}=0$.
Thus
\begin{equation}\label{Ck_reduced}
	\partial_t^2C_k-\Delta C_k=-\int _0^{t}\int _{0}^{s}\nabla\div F_k(u,x)\d u\d t  +F_k
\end{equation}
with vanishing initial conditions, which
can be solved for $C_k$ using Duhamel's formula.
Moreover,  the solution will be smooth and compactly supported in space for $t\in [0,T]$, because the source terms are.

After solving \eqref{Ck_reduced}, we need to check that $C_k$ indeed satisfies \eqref{eq:Ck_eq}. 
Apply $\partial_t^{2}\div $ to both sides of \eqref{Ck_reduced}: this yields
\begin{equation}
	(\partial_t^{2}-\Delta)(\partial_t^{2}\div C_k-\div F_k)=0,
\end{equation}
where $\partial_{t}^{2}\div C_k-\div F_k$ vanish for $0\leq t\ll 1$.
By standard uniqueness of solutions for the scalar wave equation, $\partial_{t}^{2}\div C_k-\div F_k=0$ for all time, so that $\div C_k=\int _0^{t}\int _{0}^{s}\div F_k(u,x)\d u\d t$. Plugging this into \eqref{Ck_reduced} we obtain \eqref{eq:Ck_eq}.

By last statement in Lemma \ref{lm:C0} combined with the discussion above, the first two zero harmonics $C_0$ and $C_1$ vanish.
It follows from \eqref{eq:F}, upon noticing that when $\ell=0,1$, any product of $\u E_j$, $\partial_t^2U_\ell$ involving exactly  one factor of the latter contributes nothing to the zero mode, that the zero harmonic $C_2$ at order $O(h^{7/2})$ solves 
\begin{equation}\label{eq:first_zero_harm}
\begin{aligned}
		 \partial_t^2C_2-\Delta C_2+\nabla \div C_2&=-\chi^{(3)}\Big(2|\partial_t U_0|^2\u E_0+2(U_0\cdot \partial_t^{2} U_0)\u E_0 \\
		  & \qquad +2 (\u E_{0}\cdot \partial_t^{2} U_0)U_0
		   +4(\u E_{0}\cdot \partial_t U_0)\partial_t U_0+2(\u E_{0}\cdot U_0)\partial_t^{2}U_0\Big)^{\{0\}},
\end{aligned}
\end{equation}
with vanishing initial conditions.
% The presence of a harmonic at this level is the result of the interaction of the $+1$ and $-1$ modes of $U_0$, combined with an instance of the strong electric field.
Observe that if the strong field $\underline{E}_0$ vanishes or if $\chi^{(3)}=0$,  then no zero harmonic $C_2$ is generated at this order.

\smallskip

Combining Lemmas \ref{lm:stationary}-\ref{lm:C0}, we see how to construct an asymptotic solution of \eqref{eq:second_order_metric3} for $t\in [0,T]$ up to order $O(h^{k+1/2})$, $k\geq 0$, subject to \eqref{IC1}:
first construct $\u{E}_j$ for $0\leq j\leq k+1$ with the help of Lemma \ref{lm:stationary}, the profiles of $U_{\mathrm{in}}$ for small positive time (before they interact with the nonlinearity) and $U_{\mathrm{out}}$ for $t\in [0,T]$, up to order $O(h^{k+1})$.
Then use Lemma \ref{lm:eikonal_transport} to construct $U_0$ up to the zero mode as described in Section \ref{sub:the_asymptotic_solution_the_leading_term}.
For the zero mode of $U_0$, as already explained, Lemma \ref{lm:C0} together with vanishing initial conditions implies $C_0=0$. 
Then solve  \eqref{eq:U_1_first}--\eqref{eq:U_1_second} for $k=0$ with the initial conditions coming from the $O(h)$ term of $U_{\mathrm{in}}$ at small positive time, to determine the components of $\partial_\theta U_1$, that is, $U_1$ up to its zero mode. 
Note here that \eqref{eq:U_1_second} is solvable with the process explained after Lemma \ref{lm:U1} because its zero mode vanishes, since $C_0=0$ satisfies \eqref{eq:Ck_eq}.
The zero mode of $U_1$ vanishes as well, as already explained.
Then proceed inductively using Lemmas \ref{lm:U1}--\ref{lm:C0} to construct $U_\ell$ for $\ell\leq k$ and $\partial_\theta U_{k+1}$.
This proves the following:

\begin{proposition}\label{prop:approx_sol}
	Let $k\geq 0$ and $T>{}0$. There exists a formal asymptotic solution of \eqref{eq:second_order_metric3}  up to order $O(h^{k+1/2})$ of the form $ E_{\rm in}+h^{3/2}U_{\rm out}$, with $ E_{\rm in}$ as in \eqref{ansatz} and $U_{\rm out}$ as in \eqref{Uout'}, defined on $ [0,T]\times \R^3$, subject to initial conditions \eqref{IC1}. Moreover, $\underline{E}_0$ is constant, and for each $j\geq 0$, $\underline{E}_{j+1}$ and $U_j$ are smooth and compactly supported in space.
\end{proposition}

\begin{remark}
	The zeroth harmonics $h^{7/2 }C_2(t,x)+ h^{9/2}C_3(t,x)+\dots$ 
	can be considered as an effect of the ``beam''  on the stationary electric field $h^{1/2}\underline E$. Shining the beam through the strong stationary field affects the beam by a nonlinear interaction, but it also modifies $h^{1/2}\underline E$ by an $O(h^{7/2})$ term.
\end{remark}

\section{The strong electric field \texorpdfstring{$\underline{E}$}{E} }\label{eq:choosing_the_strong_electric_field}

While a pair of constant fields $(E,B)$ satisfies the Maxwell system without the divergence free conditions, it does not solve the latter, since in general $\div D=\div (\o{\chi}^{(3)}|E|^2E)\not=0$ for $E$ constant. 
In our ansatz \eqref{A1'} we included a time independent field $h^{1/2}\underline{E}(x,h)$, and we would like it, together with a time-independent magnetic field, to be a stationary solution of the Maxwell system. 

We start with a brief remark about the stationary solutions $(E_{(0)},H_{(0)})$ of \eqref{eq1a}--\eqref{eq4a}, \eqref{IC_EH}, and respectively of \eqref{eq:second_order_metric3}, \eqref{IC_wave}, each satisfying the corresponding divergence free conditions. The initial conditions for \eqref{eq1a}--\eqref{eq4a} are the fields  themselves at $t=0$; for \eqref{eq:second_order_metric3}, they are $(E_{(0)},E'_{(0)})$ with $E'_{(0)}$ determined by \eqref{IC_E2}, see Proposition~\ref{pr1}. The Maxwell system for $(E_{(0)},H_{(0)})$ is decoupled, and we described the stationary solutions $(0,H_{(0)})$ earlier: they are $H_{(0)}=\nabla\phi$ with $\Delta\phi=0$.  A similar argument, which we make more explicit below, describes the set of the possible $E_{(0)}$ as $E_{(0)} = \n\psi$ with $\psi$ solving a nonlinear elliptic PDE, see \eqref{eq:el} below, with $h=1$ there formally. 

By \eqref{eq2a} and \eqref{div0}, $h^{1/2}\underline{E}(x,h)$ must solve
\be{B1}
\curl h^{1/2}\underline{E}=0, \quad \div\left(h^{1/2}\underline{E}+ \o{\chi^{(3)}}h^{3/2}|\underline{E}|^2\underline{E}\right)=0.
\ee
The first equation suggests setting $\underline{E}=\nabla \psi$, where $\psi$ is not necessarily small at infinity (in fact, we would like it %
to satisfy an arbitrary regular enough boundary condition at the boundary of a large domain).  In the second equation, we eliminate the factor $h^{1/2}$ from both terms and reach 
\begin{equation}   \label{eq:el}
\div \left(\nabla\psi+ h\o{\chi} |\nabla\psi|^2\nabla\psi\right) =0.
\end{equation}
Here and for the rest of this section we write $\chi$ instead of $\chi^{(3)}$ to avoid overburdening the notation.
Equation \r{eq:el} is the Euler-Lagrange equation for the double-phase functional
\[
I\coloneqq  \int \Big(\frac12 |\nabla\phi|^2+ \frac14h\chi |\nabla\phi|^4  \Big)\d x,
\]
for which properties of minimizers have been extensively studied in the literature, see e.g., \cite{Lad-Ural}, and \cite{BaroniColomboMingione} and the references there.
Here we will take a direct approach (for which the smallness of $h$ is essential) and construct a sufficiently regular solution of \eqref{eq:el} in a large domain using the Banach fixed point theorem.

Let $M\subset \R^3$ be a large domain with smooth boundary  containing the support of $\chi$.
Also let $s>5/2$,  which guarantees that $H^{s-1}(M)$ is an algebra, see \cite{Grisvard, Behzadan}.
We seek a solution to \eqref{eq:el} with the boundary condition
\begin{equation}\label{eq:bd_cond_M}
\psi\big|_{\partial M}= f\in H^{{s-1/2}}(\partial M).
\end{equation}
Let $u_f\in H^{s}(M)\cap C^\infty(M)$ be the harmonic extension of $f$ to $M$, that is, the solution to
\begin{equation}
	\Delta u=0\text{ on }M, \quad u\big|_{\partial M}=f, 
\end{equation}
see \cite[Theorem~4.21]{McLean-book}.
With respect to the ${f}$ above, and for a fixed $\delta>0$, consider
\begin{equation}
	X_{f,\delta} \coloneqq \{u\in H^{s}(M):\|u-u_f\|_{s}\leq \delta \}.
\end{equation}
This is a  non-empty closed subset of a complete metric space, hence it is complete with the metric induced on it by the $H^s$ norm.
Define
\begin{equation}
	T_{f,\delta,h}: X_{f,\delta}\to H^{s}(M), \quad T_{f,\delta,h}(u)=u_f-h\Delta_D^{-1}(\div (\chi |\nabla u|^2\nabla u)),
\end{equation}
where the codomain property uses the fact that $H^{s-1}$ is an algebra.

\begin{lemma}\label{lm:contraction}Fix $s>5/2$. There exists $h_0>0$ (depending on $\chi$, $\delta$, $f$, $s$ and $M$) such that for $0<h\leq h_0$, $T_{f,\delta,h}: X_{f,\delta}\to X_{f,\delta}$ is a contraction.
\end{lemma}
\begin{proof}
First observe that by the algebra property of $H^{s-1}(M)$  there exist constants $C_0,$ $C_1$ depending on $M$ and $s$ such that 
\begin{align}
	\left\|\Delta_D^{-1}\div\big(\chi \langle\nabla \varphi_1,\nabla\varphi_2\rangle \nabla \varphi_3\big) \right\|_{s}
	\leq &C_0\|\chi\|_{s-1}\|\nabla\varphi_1\|_{s-1}\|\nabla\varphi_2\|_{s-1}\|\nabla\varphi_3\|_{s-1}\label{eq:C0}\\
		\leq & C_1\|\chi\|_{s-1}\|\varphi_1\|_s\|\varphi_2\|_s\|\varphi_3\|_s,\quad \varphi_j\in H^s(M).\label{eq:C1}
\end{align}
In particular, if $h_0$ is such that
\begin{equation}
	C_1h_0\|\chi\|_{s-1}\big(\|u_f\|_s+\delta\big)^{3}\leq \delta,
\end{equation}
then for $0<h\leq h_0$ we have $\| T_{f,\delta,h}(u)-u_f\|_s\leq \delta  $ for all $ u\in X_{f,\delta}$, which implies $T_{f,\delta,h}: X_{f,\delta}\to X_{f,\delta}$.

Shrinking $h_0$ further if necessary to ensure that
\begin{equation}\label{eq:h_0}
	3C_1h_0\|\chi\|_{s-1}\big(\|u_f\|_s+\delta\big)^{2}\leq  1/2,
\end{equation}
we have, for $u$, $u'\in X_{f,\delta}$ and $0<h\leq h_0$,
\begin{equation}\label{eq:contraction_half}
	\begin{aligned}
	\|T_{f,\delta,h}(u)	-T_{f,\delta,h}(u')\|_s	&\overset{\eqref{eq:C1}}{\leq} hC_1\|\chi\|_{s-1}\left(\|u\|_{s}^2
	 +\| u'\|_{s}\|u\|_{s}+\| u'\|_{s}^2\right)	\|  u'- u\|_{s}\\
	 &\quad \leq 3 hC_1\|\chi\|_{s-1}(\|u_f\|_{s}+\delta)^2	\|  u'- u\|_{s}
	 \leq \frac{1}{2} \| u'-u\|_s, 
\end{aligned}
\end{equation}
that is, $T_{f,\delta,h}$ is a contraction.
\end{proof}

Using Lemma \ref{lm:contraction} we obtain existence and uniqueness of solutions to \eqref{eq:el}--\eqref{eq:bd_cond_M} in $X_{f,\delta}$, as well as an asymptotic expansion in $h$:

\begin{theorem}\label{thm:exp_psi}
	Fix $s>5/2$ and $\delta>0$. There exists $h_0>0$, depending on $M$, $\chi$, $\delta$, $f$, and $s$, such that for each $0<h\leq h_0$ there exists a unique solution $\psi(\cdot,h) \in X_{f,\delta}$ to \eqref{eq:el} with the boundary condition \eqref{eq:bd_cond_M}.
	Moreover,  $\psi$ admits an asymptotic expansion in $h$ in the sense that for any fixed $n\geq 0$ there exist $\psi^{(j)}\in H^{s}(M)$, $0\leq j\leq n $, and $R_{n+1}:(0,h_0]\to H^s(M)$  such that 
	\begin{equation}\label{eq:psi_as}
	\psi = \sum _{j=0}^n h^j\psi^{(j)}+R_{n+1}(h), \quad R_{n+1}(h)= O_{H^s(M)}(h^{n+1}), \quad \text{for all }0<h\leq h_0.
\end{equation}
One has $\psi^{(0)}=u_f$, $\psi^{(1)}=-\Delta_D^{-1}(\div (\chi |\nabla u_f|^2\nabla u_f))$. 
If $f\in C^{\infty}(\partial M)$, then $\psi^{(j)}\in C^{\infty}(\overline M)$ for all $0\leq j\leq n$.
\end{theorem}
\begin{proof}

By the Banach fixed point theorem, and with $h_0$ as in Lemma \ref{lm:contraction}, for each $0<h\leq h_0$, $T_{f,\delta,h}$ admits  a unique fixed point $\psi\in X_{f,\delta} $. 
Since such a $\psi$ is a fixed point of $T_{f,\delta,h}$ if and only if it solves \eqref{eq:el}--\eqref{eq:bd_cond_M}, we conclude that there exists a unique solution to \eqref{eq:el}--\eqref{eq:bd_cond_M} in $X_{f,\delta}$.

To show the statement regarding the asymptotic expansion, fix $n\geq 0$.
Again by the Banach fixed point theorem, the solution $\psi$ arises as the limit of the convergent sequence
\begin{equation}
		\psi_0=u_f,\qquad \psi_{k+1}:=T_{f,\delta,h}\psi_k = u_f-h\Delta_D^{-1}(\div (\chi |\nabla \psi_k|^2\nabla \psi_k)), \quad k\geq 0,
\end{equation}
which can be written as 
\begin{equation}
		\psi = \lim_{n\to \infty} \psi_n=\psi_0+\lim _{n\to \infty}\sum _{k=0}^{n-1}(\psi_{k+1}-\psi_k).
\end{equation}
Using \eqref{eq:contraction_half} for $u=\psi_{k}$, $u'=\psi_{k-1}$ we can see that for every $k\geq 1$, 
\begin{equation}\label{eq:n+2-n+1}
\begin{aligned}
	\|\psi_{k+1}-\psi_{k}\|_s{\leq} & 3C_1h\|\chi\|_{s-1}\big(\|u_f\|_s+\delta\big)^2\|\psi_{k}-\psi_{k-1}\|_{s}\\ 
	{\leq} & h^{k}\Big(3C_1\|\chi\|_{s-1}\big(\|u_f\|_s+\delta\big)^2\Big)^{k}\|\psi_{1}-\psi_0\|_{s}.
\end{aligned}
\end{equation}
Using this and the estimate $\| \psi_1-\psi_0\|_s\leq hC_1\|\chi\|_{s-1}\|\psi_0\|_{s}^{3}$, which follows by \eqref{eq:C1}, we see that
\begin{equation}\label{eq:estimate_psi}
	\begin{aligned}
	\|\psi-\psi_n\|_s\leq \sum _{k=n}^\infty\|\psi_{k+1}-\psi_k\|_s\overset{\eqref{eq:n+2-n+1}}{\leq }\sum _{k=n}^\infty  h^{k}\big(3C_1\|\chi\|_{s-1}\big(\|u_f\|_s+\delta\big)^2\big)^{k}\|\psi_{1}-\psi_0\|_{s}\\
	\leq
	h^{n+1}C_1\|\chi\|_{s-1}\|\psi_0\|^{3}\sum _{m=0}^\infty \big(3C_1\|\chi\|_{s-1}\big(\|u_f\|_s+\delta\big)^2\big)^{m+n}h_{0}^{m},
\end{aligned}	
\end{equation}
where the infinite sum converges by \eqref{eq:h_0}. 
Substituting inductively $\psi_k$ for $k<n$ in $\psi_n$, we see that $\psi_n$ is a polynomial of degree $d_n$ in $h$, where $d_n$ is given  recursively as $d_n=3d_{n-1}+1$, $d_0=0$.
In particular, since $d_n\geq n$, for $0\leq j\leq n$ there exist $h$-independent $\psi^{(j)}\in H^{s}(M)$ such that 
\begin{equation}
	\psi_n = \sum _{j=0}^{n} h^j\psi^{(j)}+O_{H^s(M)}(h^{n+1}).
\end{equation}
Combined with \eqref{eq:estimate_psi}, this shows \eqref{eq:psi_as}. 
The statement about $\psi^{(0)}$ and $\psi^{(1)}$ can be read off from $\psi^{1}$.

Finally, if $f$ is smooth, then $u_f\in C^\infty(\overline{M})$, so from the mapping property of $(\Delta_D)^{-1}:H^{s-2}\to H^s$ it follows that $\psi_k\in C^{\infty}(\overline{M})$ for all $k\geq 0$. Thus this is also true for the $\psi^{(k)}$. 
\end{proof}

\smallskip

We can proceed to the following proof now. 

\begin{proof}[Proof of Theorem \ref{thm_main}\ref{item:a} and Lemma \ref{lm:stationary}]
Applying Theorem~\ref{thm:exp_psi} with $M=B(0,R_0)$ and setting $\underline{E}=\nabla \psi$  proves Theorem~\ref{thm_main}\ref{item:a}, except for uniqueness.
For that, if  $E\in H^{s-1}(M)$ as in the statement of the theorem satisfies $\curl E=0$, a low regularity version of the Poincaré lemma (e.g., \cite[Theorem 8.3]{Csato2012}) implies the existence of $\psi\in H^{s}(M)$ such that $E= \nabla \psi $.
This $\psi$ satisfies \eqref{eq:el} and $\psi-f\big|_{\partial M}=a=\textrm{const}$. 
Using \eqref{eq:C0},
\[
		 \|\psi-u_f-a \|_s =\| h\Delta_D^{-1}(\div (\chi |\nabla \psi|^2\nabla \psi))\|_s{\leq }C_0h\|\chi\|_{s-1}\|\nabla\psi\|_{s-1}^{3}\\
		\leq C_0h\|\chi\|_{s-1}(\|\nabla{u_f}\|_{s-1}+\epsilon)^{3}.		%
 \]
Taking $E_j=\nabla \psi_j$, $j=1,2$ with the properties above we see that if $h$ is sufficiently small, then $\psi_j-a_j\in X_{f,\delta}$ and solve \eqref{eq:el}, \eqref{eq:bd_cond_M}, so they must be equal, that is, $\psi_1-a_1=\psi_2-a_2\implies E_1=E_2$.

To prove Lemma \ref{lm:stationary}, use Theorem~\ref{thm:exp_psi} again, with $f= a x_3$, where $a$ is constant.
We will extend $\psi$ suitably to $\R^3$. 
Fix $n\geq 0$; note that for $0\leq j\leq n$ we have  $\psi^{(j)}\in C^\infty(\overline M)$ in \eqref{eq:psi_as}.
Denote by $D$ a bounded domain with $\overline{M} \subset D$,  extend $\psi^{(0)}$ as $\tilde{\psi}^{(0)}=ax_3$ on $\R^{3}$, and each $\psi^{(j)}$ to a smooth function $\tilde{\psi}^{(j)}\in C_0^\infty(D)$.
As for the error term, by \eqref{eq:psi_as} there exists a bounded function $(0,h_0]\to H^{s}(M)$, $h\mapsto h^{-n-1}R_{n+1}(h)$.  
Using standard extension operators for Sobolev functions on smooth, bounded domains and a suitable cutoff, there exists a bounded operator $\mathcal{E}:H^{s}(M)\to H^{s}_{\overline{D}}(\R^3)$, therefore upon setting $\tilde{R}_{n+1}=h^{n+1}\mathcal{E}(h^{-n-1}R_{n+1}(h))$ we obtain a map $\tilde{R}_{n+1}:(0,h_0]\to H^{s}_{\overline{D}}(\R^3)$  with the property $\tilde{R}_{n+1}=O_{H^{s}(\R^3)}(h^{n+1})$.
Letting $\tilde{\psi}= \sum _{j=0}^n h^j\tilde{\psi}^{(j)}+\tilde{R}_{n+1}(h)$, we obtain a function in $H^{s}(\R^3)$ that solves \eqref{eq:el} on $M\cup \overline{D}^c$.
Now set
\begin{equation}
	h^{1/2}\underline{E}(\cdot ,h)=h^{1/2}\nabla\psi\in H^{s-1}_{\mathrm{loc}}(\R^n),
\end{equation}
which satisfies $\curl \u E=0$ on $\R3$, the divergence free condition \eqref{div0}  in $M\cup \overline{D}^c$,
 and has an asymptotic expansion in $h$, in the sense that for every $n\geq 0$,
\begin{equation}\label{eq:E_error}
	h^{1/2}\underline{E}(x,h)-h^{1/2}\sum_{j=0}^nh^j\u{E}_j=O_{H^{s-1}(\R^3)}(h^{n+3/2}), \quad \u{E}_j\coloneqq \nabla \tilde{\psi}^{(j)}.
\end{equation}
Finally, since the $O(1)$  term of $\u{E}$ is bounded, the next $n-1$ are in $C_0^\infty$, and the error is in $H^{s-1}(\R^3)\subset L^\infty(\R^3)$ (by Sobolev Embedding), we see that $\u{E}\in L^\infty(\R^3)$.
\end{proof}

\section{The Profile Equations -- Proofs of Lemmas \ref{lm:eikonal_transport}-\ref{lm:C0}}\label{sec:the_profile_equations_proofs_of_lemmas_ref_lm_eikonal_transport}

In this section we will prove Lemmas \ref{lm:eikonal_transport}-\ref{lm:C0} by substituting  \eqref{ansatz} into \eqref{eq:second_order_metric3}.
This is the general approach for the construction of asymptotic solutions of symmetric hyperbolic systems (see e.g., \cite{Metivier-Notes}), adapted to a second order system.
For the duration of this section we write $E$ instead of $ E_{\rm in}$ to avoid over-cluttering the notation.

\begin{proof}[Proof of Lemma \ref{lm:eikonal_transport}]

Substitute \eqref{ansatz} into \eqref{eq:second_order_metric3}.
The leading order of the left hand side of \eqref{eq:second_order_metric3} is $O(h^{-1/2})$.
By
 \begin{equation}   \label{Ett}
\partial_t^2 (| E|^2 E) = 2|\partial_t E|^2 E+ 2 ( E\cdot   \partial_t^2 E)  E +4( E\cdot   \partial_t E) \partial_t E 
+ | E|^2 \partial_t^2 E,
\end{equation}
the right hand side does not contribute to order $O(h^{-1/2})$, since
the terms there are of order $h^{3/2}$, $ h^{1/2}$, $h^{3/2}$, $ h^{1/2}$, respectively. 
Thus the coefficient of $h^{-1/2}$ is only determined by the left hand side, and it must vanish; we obtain
\begin{equation}
	((\partial_t \phi)^2-|\nabla \phi|^2 )\partial_\theta^2U_0+(\nabla \phi\cdot \partial_\theta^2U_0)\nabla \phi =0.
\end{equation}
It follows that $\phi =- t+\omega\cdot x$ is a characteristic phase of multiplicity 2, in the sense that it satisfies
\begin{equation}
	\det \left( ((\partial_t \phi)^2-|\nabla \phi|^2 )\Id_3+\nabla \phi\otimes \nabla \phi\right)=0,
\end{equation}
and for every $(t,x)$, the dimension of the nullspace of $((\partial_t \phi)^2-|\nabla \phi|^2 )\Id_3+\nabla \phi\otimes \nabla \phi$ is 2 (and given by $\omega^{\perp}$). 
Then  $U_0$ satisfies the polarization condition 
\begin{equation}
	\partial_\theta^{2}U_0\in \ker \left( ((\partial_t \phi)^2-|\nabla \phi|^2 )\Id_3+\nabla \phi\otimes \nabla \phi\right)\iff 
	\partial_\theta^{2}U_0\cdot \omega=0.\label{eq:polarization}
\end{equation}
Since $U_0$ is assumed to be periodic in $\theta$, this condition is equivalent to $ \partial_\theta U_0\cdot \omega=0$, i.e., \eqref{tE0}.
This also proves the last statement, that is, if \eqref{tE0} holds then \eqref{eq:second_order_metric3} is true to order $O(h^{-1/2})$.

At order $h^{1/2}$, the right hand side enters and we obtain
\begin{equation}   \label{tE'}
\begin{aligned}
	-2(\partial_t + \omega\cdot{\nabla})& \p_\theta U_0+ \chi^{(3)}  |\u{E}_0|^2 \p_{\theta}^2U_0 + 2\chi^{(3)}   (\u{E}_0\cdot \p_{\theta}^2U_0 )\u{E}_0 \\
	&+(\div \partial_\theta U_0)\omega+{\nabla} (\partial_\theta U_0(t,x,\theta)\cdot \omega )+(\partial_\theta^2U_1\cdot \omega)\omega=0.
\end{aligned}
\end{equation}
Notice that \eqref{tE'} is a linear equation for $\partial_\theta U_0$.
It implies in particular that the first five terms must be in the range of the last, which is spanned by $\omega$.
That is, they must be annihilated by $\Pi_\omega^\perp v =v-(v\cdot \omega)\omega$, the projection on $\omega^{\perp}$.
Applying $\Pi_\omega^\perp$ to \eqref{tE'} we obtain
\begin{equation}
	-2(\partial_t + \omega\cdot{\nabla}) \Pi_\omega^\perp\p_\theta U_0+ \chi^{(3)}  |\u{E}_0|^2 \Pi_\omega^\perp \p_{\theta}^2U_0 + 2\chi^{(3)}   (\u{E}_0\cdot \p_{\theta}^2U_0 )\Pi_\omega^\perp \u{E}_0 
	+\Pi_\omega^\perp{\nabla} (\partial_\theta U_0(t,x,\theta)\cdot \omega )=0.
\end{equation} 
Using that $\omega\cdot \u{E}_0=0$ and \eqref{eq:polarization}, we reach
\begin{equation}
	-2(\partial_t + \omega\cdot{\nabla}) (\Pi_\omega^\perp\p_\theta U_0)+ \chi^{(3)} |\u{E}_0|^2 \partial_\theta (\Pi_\omega^\perp \p_{\theta}U_0) + 2\chi^{(3)}   (\u{E}_0\cdot \partial_\theta (\Pi_\omega^\perp \p_{\theta} U_0) ) \u{E}_0 =0,
\end{equation}
i.e., \eqref{tEp}.
\end{proof}

\begin{proof}[Proof of Lemma \ref{lm:U1}.]
Recall our assumption that $\partial_\theta U_0\cdot \omega=0$, $\u{E}_0$ is constant with $\u{E}_0\cdot \omega =0$, and  $ \phi=-t+\omega\cdot x$.
Also note that, by  Lemma \ref{lm:stationary}, $-\Delta \u{E}_{k+1}+\nabla \div \u E_{k+1}=\curl \curl \u E_{k+1}=0$ for all $k\geq 0$.
We show \eqref{eq:U_1_second}.
Write the linear part of \eqref{eq:second_order_metric3} as 
\begin{equation}\label{eq:linear_part}
	LE\coloneqq \partial _t^2 E-\Delta E+\nabla \div E .
\end{equation}
Let  $E$ be as in \eqref{ansatz}. 
 We find
\begin{equation}   \label{tE''}
\begin{aligned}
	LE&\sim h^{-1/2}(\partial_\theta^2U_{0}\cdot \omega)\omega\\
	&\quad {}+h^{1/2}\Big( -2(\partial_t + \omega\cdot{\nabla}) \p_\theta U_0
	+( \div \partial_\theta U_0)\omega+{\nabla}( \partial_\theta U_0(t,x,\theta)\cdot \omega)+(\partial_\theta^{2}U_1\cdot \omega)\omega\Big)   \\
	&  \quad \quad +\sum_{k\geq 0} h^{3/2+k}\Big(-2(\partial_t + \omega\cdot{\nabla}) \p_\theta U_{k+1} +\partial_t^2U_{k}-\Delta U_{k} +\nabla \div U_{k}\\
	&\qquad\qquad\qquad+( \div\partial_\theta U_{k+1})\omega+{\nabla} (\partial_\theta U_{k+1}(t,x,\theta)\cdot \omega)+(\partial_\theta^2U_{k+2}\cdot \omega)\omega\Big).%
\end{aligned}
\end{equation}
Write the right hand side of \eqref{eq:second_order_metric3} as
\begin{equation}\label{eq:forms}
\begin{aligned}
		 -2\chi^{(3)}|\p_tE|^2E- 2\chi^{(3)} (E\cdot &  \p_t^2E) E -4\chi^{(3)}(E\cdot   \p_tE) \p_tE 
- \chi^{(3)}|E|^2 \p_t^2E\\ 
=&\sum_{j=1}^2\omega^{[1]}_j(\partial_t^2 E,E,E)+\sum_{j=1}^2\omega_j^{[2]}(\partial_t E,\partial_t E, E) ,
\end{aligned}	 
\end{equation}
where $\omega_j^{[\ell]}$ are  $\R^3$-valued tensors. 
We have, with $\phi=-t+\omega\cdot x$,
\begin{equation}\label{eq:E_der}
\begin{aligned}
		\partial_tE
		\sim &-h^{1/2}\partial_\theta U_0 + \sum_{k\geq 0} h^{3/2+k}\left(-\partial_\theta U_{k+1}  +\partial_t U_k\right),\\
		\partial_t^2E 
		\sim &h^{-1/2}\partial_\theta^{2} U_0 +h^{1/2}\left( -2\partial_{t\theta}^2 U_0  +\partial_\theta^{2}U_{1} \right)
		+\sum _{k\geq 0} h^{3/2+k}\left( -2\partial_{t\theta}^2 U_{k+1}  
		+\partial_t^2U_k+\partial_\theta^{2}U_{k+2}  \right).
\end{aligned}
\end{equation}
Denote by $A^{(\lambda)}$ the coefficient of $h^{\lambda}$  in the expansion of an expression $A$.
Equating \eqref{tE''} and \eqref{eq:forms} at order $h^{3/2+k}$, $0\leq k\leq K$, 
\begin{equation}\label{eq:O(h3/2+k)}
	\begin{aligned}
			-2(\partial_t &+ \omega\cdot{\nabla}) \p_\theta U_{k+1} +\partial_t^2U_{k}-\Delta U_{k} +\nabla \div U_{k}\\*
			&+( \div\partial_\theta U_{k+1})\omega+{\nabla} (\partial_\theta U_{k+1}(t,x,\theta)\cdot \omega)+(\partial_\theta^2U_{k+2}\cdot \omega)\omega\\*
			&\quad =\sum_{j=1}^2\sum _{s=-1}^{k}\sum_{t=0}^{k-s}\omega^{[1]}_j\left((\partial_t^2 E)^{\left(s+\frac{1}{2}\right)},E^{\left(t+\frac{1}{2}\right)},E^{\left(k-s-t+\frac{1}{2}\right)}\right)\\*
			&\qquad + \sum_{{j=1}}^2\sum _{s=0}^{k}\sum_{t=0}^{k-s}\omega^{[2]}_j\left((\partial_t E)^{\left(s+\frac{1}{2}\right)},(\p_tE)^{\left(t+\frac{1}{2}\right)},E^{\left(k-s-t+\frac{1}{2}\right)}\right).
	\end{aligned}
\end{equation}

We now argue that $U_{\ell}$ with $\ell\geq k+2$ does not appear in the right hand side of \eqref{eq:O(h3/2+k)}, and $U_{k+1}$ only appears linearly.
From \eqref{ansatz} and \eqref{eq:E_der} follows that $U_{k+2}$ and derivatives of it only appear in $(\partial_t^{n}E)^{(7/2+k-l)}$ for $l\leq n$ and $0\leq n\leq 2$.
It is now easy to check that it cannot appear in the right hand side of \eqref{eq:O(h3/2+k)} due to the ranges of the summations.
Similarly, since $U_{k+1}$ appears in 
$(\partial_t^{n}E)^{(5/2+k-l)}$ for $l\leq n$ and $0\leq n\leq 2$,
it only shows up in the right hand side of \eqref{eq:O(h3/2+k)} in the terms $\omega^{[1]}_j\!\big((\partial_t^2 E)^{\left(k+\frac{1}{2}\right)},E^{\left(1/2\right)},E^{\left(1/2\right)}\big)$.
We can isolate the $s=k$ term in the first summations of \eqref{eq:O(h3/2+k)} and rewrite the latter in the form 
	\begin{align} 
			 -2(\partial_t &+ \omega\cdot{\nabla}) \p_\theta U_{k+1} +\partial_t^2U_{k}-\Delta U_{k} \\*
			&+\nabla \div U_{k}+( \div\partial_\theta U_{k+1})\omega+{\nabla}( \partial_\theta U_{k+1}(t,x,\theta)\cdot \omega)+(\partial_\theta^2U_{k+2}\cdot \omega)\omega\\*
			&\quad=-\chi^{(3)} |\u{E}_0|^2(\partial_\theta^2 U_{k+1})-2\chi^{(3)} (\u{E}_0\cdot \partial_\theta^2 U_{k+1}) \u{E}_0 \label{eq:big_eq_2}\\*
			&\quad-\chi^{(3)} |\u{E}_0|^2(-2\partial_{t\theta}^2U_k+\partial_t^2U_{k-1})-2\chi^{(3)} (\u{E}_0\cdot (-2\partial_{t\theta}^2U_k+\partial_t^2U_{k-1})) \u{E}_0\\*
			&\qquad +\sum_{j=1}^2\sum _{s=-1}^{k-1}\sum_{t=0}^{k-s}\omega^{[1]}_j\left((\partial_t^2 E)^{\left(s+\frac{1}{2}\right)},E^{\left(t+\frac{1}{2}\right)},E^{\left(k-s-t+\frac{1}{2}\right)}\right)\\*
			&\qquad\quad + \sum_{j=1}^2\sum _{s=0}^{k}\sum_{t=0}^{k-s}\omega^{[2]}_j\left((\partial_t E)^{\left(s+\frac{1}{2}\right)},(\p_tE)^{\left(t+\frac{1}{2}\right)},E^{\left(k-s-t+\frac{1}{2}\right)}\right),			
	\end{align}
where for convenience we set $U_m=0$ if $m<0$.
Apply $\Pi_\omega^{\perp}$  to find 
	\begin{align}
			 -2(&\partial_t + \omega\cdot{\nabla}) \p_\theta \Pi_\omega^{\perp}U_{k+1} +\Pi_\omega^{\perp}(\nabla( \partial_\theta U_{k+1}(t,x,\theta)\cdot \omega))\\*
			 \quad&+\chi^{(3)} |\u{E}_0|^2(\partial_\theta^2 \Pi_\omega^{\perp}U_{k+1})+2\chi^{(3)} (\u{E}_0\cdot \partial_\theta^2 \Pi_\omega^{\perp} U_{k+1}) \u{E}_0+\Pi_\omega^{\perp}(\nabla \div U_{k}+\partial_t^2U_{k}-\Delta U_{k} )\\*
			&\label{eq:perp_part}\quad=
			-\chi^{(3)} |\u{E}_0|^2\Pi_\omega^{\perp}(-2\partial_{t\theta}^2U_k+\partial_t^2U_{k-1})-2\chi^{(3)} (\u{E}_0\cdot (-2\partial_{t\theta}^2U_k+\partial_t^2U_{k-1})) \u{E}_0\\*
			&\qquad +\sum_{j=1}^2\sum _{s=-1}^{k-1}\sum_{t=0}^{k-s}\Pi_\omega^{\perp}\omega^{[1]}_j\left((\partial_t^2 E)^{\left(s+\frac{1}{2}\right)},E^{\left(t+\frac{1}{2}\right)},E^{\left(k-s-t+\frac{1}{2}\right)}\right)\\*
			&\qquad\quad + \sum_{j=1}^2\sum _{s=0}^{k}\sum_{t=0}^{k-s}\Pi_\omega^{\perp}\omega^{[2]}_j\left((\partial_t E)^{\left(s+\frac{1}{2}\right)},(\p_tE)^{\left(t+\frac{1}{2}\right)},E^{\left(k-s-t+\frac{1}{2}\right)}\right).			
	\end{align}
This equation can be more succinctly expressed in the form \eqref{eq:U_1_second}.

For \eqref{eq:U_1_first},
applying $\Pi_\omega$ in \eqref{eq:big_eq_2} with $k$ replaced by $k-1$, we obtain an equation of the form
\begin{equation}
	\partial_\theta^{2}\Pi_\omega U_{k+1}=\Pi_\omega F_k\big(\chi^{(3)}, \u{E}_j,\partial^\alpha_{(t,x,\theta)}U_j:j=0,\dots, k,|\alpha|\leq 2\big),
\end{equation}
where each component of $F_k$ is  a polynomial in the entries of its arguments.
The right hand side must average to 0 with respect to $\theta$, since the left hand side does.
We can integrate in $\theta$ and choose the integration constant (which is a function of $(t,x)$ in this case) in such a way that the right hand side has no 0 Fourier mode; this yields \eqref{eq:U_1_first}.

If $m\geq 0$, the $O(h^{m+1/2})$ term of \eqref{eq:second_order_metric3} with $E_{\rm in}$ substituted there is given by \eqref{eq:O(h3/2+k)} with $k=m-1$ (\eqref{tE'} if $m=0$), and it vanishes exactly when its projections to both $\omega$ and $\omega^{\perp}$ vanish.
This corresponds to the equations \eqref{eq:U_1_first} for $k=m$ and \eqref{eq:U_1_second} for $k=m-1$ (\eqref{tEp} if $m=0$).
Therefore, as long as these are satisfied for $m\leq K$,  \eqref{eq:second_order_metric3} is satisfied to order $O(h^{K+1/2})$.

Now suppose that there exist smooth  profiles $U_{k+1}$, $-1\leq k\leq K$, such that \eqref{tE0}--\eqref{tEp} and \eqref{eq:U_1_first}--\eqref{eq:U_1_second} hold on a time interval $[0,T]$ and that their initial conditions are those originating from \eqref{IC1}, as explained in Section \ref{sub:initialization_of_the_problem_the_linear_solution}.
This implies in particular $U_k\big|_{t=0}\in C_0^{\infty}(\R^3\times S^1)$.
We will show by induction that for every $0\leq k \leq K$, $\Pi_\omega \partial_\theta U_{k+1}$ and  $\Pi_\omega^{\perp} \partial_\theta U_{k}$ have no Fourier modes outside the range $[-k-1,k+1]$ and  lie in $C_{0}^{\infty}([0,T]\times \R^3\times S^{1})$. 
If $ k=0$, the Fourier support statement holds for $\Pi_\omega^{\perp}U_0$, since we already saw that $U_0$ contains only the modes $-1$, $0$, $1$.
The Fourier support statement for $\Pi_\omega U_{1}$, follows from \eqref{eq:U_1_first} for $k=0$: more explicitly, 
 apply $\Pi_\omega$ to \eqref{tE'} and use that $\Pi_\omega \partial_\theta U_0=\Pi_\omega \u E_0=0$ to obtain
\begin{equation}\label{eq:k=0}
	\partial_\theta^{2} \Pi_\omega U_1=-\div \partial_\theta U_0\implies \partial_\theta \Pi_\omega U_1=-\sin(\theta)\div A_0 +\cos (\theta)\div B_0, 
\end{equation}see \eqref{U00}.
Regarding spatial support, $\Pi_\omega^{\perp}\partial_\theta U_0$ solves a homogeneous transport equation with initial conditions in $C_0^{\infty}(\R^3\times S^1)$, so it lies in $C_{0}^{\infty}([0,T]\times \R^3\times S^{1})$, and $\Pi_\omega\partial_\theta U_0 =0$.
 Similarly, $\Pi_\omega \partial_\theta U_{1}\in C_0^{\infty}([0,T]\times \R^3\times S^1)$ by \eqref{eq:k=0} and \eqref{Ein}.

Assuming the statement holds for $0\leq k\leq K-1$, we will show it for $k+1$.
Eq. \eqref{eq:perp_part} is a linear transport PDE for $\Pi_\omega^{\perp} \partial_\theta U_{k+1}$ that can be split into separate transport equations corresponding to non-zero Fourier modes (the zero mode yields an equation for the zero mode of $U_k$, as in Lemma \ref{lm:C0}).
By the induction hypothesis, the source terms in the left hand side of \eqref{eq:perp_part} only contribute Fourier modes in the range $[-k-1,k+1]$.
Now look at $\omega^{[1]}_j\big((\partial_t^2 E)^{\left(s+\frac{1}{2}\right)},E^{\left(t+\frac{1}{2}\right)},E^{\left(\ell-s-t+\frac{1}{2}\right)}\big)$.
It is not hard to see from \eqref{eq:E_der} that $(\partial_t^2 E)^{\left(s+\frac{1}{2}\right)}$ contributes Fourier modes in the range $[-s-2,s+2]$ and $E^{\left(t+\frac{1}{2}\right)}$ only contributes Fourier modes in the range $[-t,t]$. Hence the maximum possible Fourier mode in the product is given by 
$	(s+2)+t+(k-s-t)=k+2$, 
similarly for the minimal one.
Similarly, $(\partial_t E)^{\left(s+\frac{1}{2}\right)} $ only contributes modes in the range $[-s-1,s+1]$.
So the maximal possible mode in the last row of \eqref{eq:perp_part} is $(s+1)+(t+1)+(k-s-t)=k+2$.
Thus, taking into account initial conditions, $\partial_\theta \Pi_\omega^{\perp}U_{k+1}$ contains no  Fourier modes outside the range $[-k-2,k+2]$.
Regarding its spatial support, the only source term in the transport equation \eqref{eq:perp_part} corresponding to non-zero Fourier modes that does not contain a factor of $\chi^{(3)}\in C_0^{\infty}(\R^3)$ is $\Pi_\omega^{\perp}(\nabla( \partial_\theta U_{k+1}(t,x,\theta)\cdot \omega))$, which is compactly supported by the inductive assumption.
Thus $\Pi ^{\perp}\partial_\theta U_{k+1}$ must also be compactly supported in space for $t\in [0,T]$, if it is so for $t=0$.
For the Fourier support of $\partial_\theta \Pi_\omega U_{k+2}$, apply $\Pi_\omega$ to \eqref{eq:big_eq_2}: exactly as before, and using in addition that the Fourier content of $U_{k+1}$ is contained in $[-k-2,k+2]$, the same must hold for $ \Pi_\omega\partial_\theta U_{k+2}$.
For its spatial support, we can now write $\partial_\theta ^{2}\Pi_\omega U_{k+2}=\sum _{0<|\ell|\leq k+2}F^{\{\ell\}}(t,x)e^{i\ell\theta}$  with $F^{\{\ell\}}\in C_{0}^{\infty}([0,T]\times \R^3)$, by \eqref{eq:perp_part} with $\Pi_\omega$ applied to it. 
Thus $ \partial_\theta \Pi_\omega U_{k+2}=\sum _{0<|\ell|\leq k+2}\frac{1}{i\ell}F^{\{\ell\}}(t,x)e^{i\ell\theta}\in C_{0}^{\infty}([0,T]\times \R^3\times S^{1})
$.
This completes the induction. 
Finally, the statement for $\Pi_\omega^{\perp}U_{K+1}$ follows exactly as that for $\Pi_\omega^{\perp}U_{k+1}$ in the inductive step.
\end{proof}

\begin{proof}[Proof of Lemma \ref{lm:C0}]
The equation for $C_k$ can be derived from the zero Fourier mode (that is, the average in $\theta$) of \eqref{eq:big_eq_2}.
The zero mode of left hand side there is 
\begin{equation}
	 \partial_t^2U_{k}-\Delta U_{k} 
			+\nabla \div U_{k}.
\end{equation}
To see what terms appear on the right hand side, write the right hand side of \eqref{eq:second_order_metric3}, with $E_{\mathrm {in}}$ substituted there, as
\begin{equation}
 	F(t,x,\theta)=-\chi^{(3)}|h^{1/2}\u E(x)+h^{3/2}U_{\rm{in}}(t,x,\theta)|^{2}(h^{1/2}\u E(x)+h^{3/2}U_{\rm{in}}(t,x,\theta)).
\end{equation}
If $\phi=-t+x\cdot \omega$,
\begin{equation}
 	\partial_{t}^2(F(t,x,\phi/h))=\partial_{t}^{2}F-2h^{-1}\partial_{t\theta}^{2}F+h^{-2}\partial_\theta^{2} F
\end{equation}
evaluated at $(t,x,\phi/h)$. 
Since $\partial_\theta$ annihilates the zero mode,
we see that only the first term on the right hand side contributes to it, and its $O(h^{k+3/2})$ term can be written as in \eqref{eq:forms} in the form 
\begin{equation}
\begin{aligned}
&
\sum_{j=1}^2\sum _{s=1}^{k}\sum_{l=0}^{k-s}\omega^{[1]}_j
\left(\partial_t^{2} U_{s-1},(\underline{E}_l+U_{l-1}),(\underline{E}_{k-s-l}+U_{k-s-l-1})\right)\\
% 			%
			&\qquad + \sum_{{j=1}}^2\sum _{s=1}^{k}\sum_{l=1}^{k-s}\omega^{[2]}_j\left(\partial_t U_{s-1},\partial_tU_{l-1},(\underline{E}_{k-s-l}+U_{k-s-l-1})\right),
\end{aligned}	 
\end{equation}
where $U_\ell=0 $ if $\ell<0$.
We thus obtain \eqref{eq:0-mode}. 

For the last statement, a computation shows that 
\begin{align}
		 &\quad \partial_t^2 F= -h^{5/2}\chi^{(3)}\big(2(\u E_0\cdot \partial_t^{2}U_{0})\u E_0+|\u E _0|^2 \partial_t^{2}U_0\big)\\* \label{eq:F}
		 &\qquad -h^{7/2	}\chi^{(3)}\Big(2|\partial_t U_0|^2\u E_0+2((\u{E}_1+U_0)\cdot \partial_t^{2} U_0)\u E_0 
		 +2(\u E_0\cdot \partial_t^2 U_1)\u E_0+|\u E_0|^2\partial_t^2 U_1\\*
		  &\qquad\qquad +2 (\u E_{0}\cdot \partial_t^{2} U_0)(\u E_1+U_0) 
		   +4(\u E_{0}\cdot \partial_t U_0)\partial_t U_0+2(\u E_{0}\cdot(\u E_1+ U_0))\partial_t^{2}U_0\Big)	 +O(h^{9/2}),
\end{align}	
hence the right hand side of \eqref{eq:big_eq_2} does not contribute to the zero mode when $k=0$.
Thus $C_0$ solves 
\begin{equation}\label{eq:C0_eq}
	\partial_t^2C_0-\Delta C_0+\nabla \div C_0  =0, \qquad C_0\big|_{t=0}=\partial_tC_0\big|_{t=0}=0,
\end{equation}
which implies $C_0=0$, as explained after the statement of Lemma \ref{lm:C0}.
Therefore, the zero mode of the $O(h^{5/2})$ term of \eqref{eq:F} vanishes and so the right hand side of \eqref{eq:0-mode} vanishes when $k=0,1$.
\end{proof}

\section{Exact Solutions }
\label{sec:exact_solutions}

In this section  we will  use a result of Guès (\cite{Gues93}) to prove the existence of exact solutions for the Maxwell system near our approximate ones (at least in our region of interest), thus justifying the formal computations above.
We consider a first order system with the property that any sufficiently $L	^\infty$-small solutions of it also solve (\ref{eq1a}--\ref{eq2a}), and vice versa. 
Recall that we assumed $c=(\mu_0\epsilon_0)^{-1/2}=1$, which amounts to replacing $t$ by $ct$. Replacing $H$ by $(\mu_0/e_0)^{1/2}H$, we arrive at (\ref{eq1a}--\ref{eq4a}) but with $\e_0$ and $\mu_0$ there replaced by 1.
Also recall that $\chi^{(1)}=0$.  
We write $(t,x)$ by $y=(y_0,y_1,y_2,y_3)\in \R^4$, and we also write $u=(E,H)^T$.
Let
\begin{equation}
	L(\chi^{{(3)}}(x),u)u\coloneqq \sum_{j=0}^3 A_j(\chi^{{(3)}}(x),u )\p_{y_j}u =0\label{first_order_modified},
\end{equation}
where
\begin{equation}\label{eq:A_0}
	A_0(a,u)=\begin{pmatrix}
		\Id_3 +\psi(a|E|^2)(a|E|^2\Id_3+2 aE\otimes E)	&\quad  0\\
		0 &\quad  \Id_3
	\end{pmatrix}\text{,}\quad
\end{equation}
with $\psi\in C^\infty(\R)$ satisfying $\psi(s)= 1$ for $s>-1/8$ and $\psi(s)=0$ for $s<-1/4$, 
and for $j=1,2,3$,
\begin{equation}
A_j(a,u)=\begin{pmatrix}
		0 & - U_j\\
		U_j & 0
	\end{pmatrix},\quad 
	U_1=\begin{pmatrix}
		0 & 0 & 0\\
		0 & 0 & -1\\
		0 & 1 & 0\\
	\end{pmatrix},
	\quad 
	U_2=\begin{pmatrix}
		0 & 0 & 1\\
		0 & 0 & 0\\
		-1 & 0 & 0\\
	\end{pmatrix},
	\quad 
	U_3=\begin{pmatrix}
		0 & -1 & 0\\
		1 & 0 & 0\\
		0 & 0 & 0\\
	\end{pmatrix}.
\end{equation} %
If $\psi$ were identically 1 everywhere, we would get exactly the original system (\ref{eq1a}-\ref{eq2a}), with $a=\chi^{(3)}$.
A solution of \eqref{first_order_modified} satisfying $\chi^{(3)} |E|^2>-1/8$ is also a solution of the original one and vice versa. 
A computation using the matrix determinant lemma shows that 
\begin{equation}
		\det(A_0-\lambda \Id_6)=(1-\lambda)^{3}\big(1+a|E|^2\psi(a|E|^2)-\lambda\big)^{2}\Big(1+3a |E|^2\psi(a|E|^2)-\lambda\Big),
\end{equation}
so the eigenvalues of $A_0$ are given by
\begin{equation}
	\lambda_1 =1,\quad \lambda_2 = 1+\psi(a|E|^2)a|E|^2,\quad \lambda_3 =1+3\psi(a|E|^2)a|E|^2.
\end{equation}
The support condition on $\psi$ guarantees that they are always positive, thus $A_0$ is symmetric positive definite.
Since $A_0^{-1}A_{j}$ is symmetric for $j=1,2,3$, \eqref{first_order_modified} is a symmetric hyperbolic system in the sense of \cite{Gues93}.
There is no zeroth order term, which implies $u=\text{const.}$ is a solution (without the divergence free condition \eqref{eq4a} taken into account). 

The initial conditions we want to impose are 
\begin{equation}\label{system_ic}
	u\big|_{t=0}=g\coloneqq \big(\ h^{1/2}\underline{E}(x,h)+2h^{3/2}U_{\rm init}(x)\cos \left(\frac{x\cdot\omega}{h}\right),\ 0\ \big)^T,
\end{equation}
with $\u E$ as in Lemma \ref{lm:stationary} with suitable $s$.
This is in accordance with taking $\partial_tE=0$ in \eqref{IC1} and the discussion following \eqref{IC_E2}. 
Note that $E\big|_{t=0}$ does \emph{not} satisfy \eqref{eq4a} everywhere.

We recall the statement of the main theorem in \cite{Gues93}, adapted to our setting.
Fix a real number $T> 0$ (this will be the time for which our asymptotic solution is defined, long enough to allow the asymptotic solution to interact with the medium and exit, so that it can be measured), and an integer $m\geq 0$; then write $\Omega= [0,T]\times \R^{3}$.
For $\rho>0$ define the following spaces of functions depending on a parameter $h\in (0,1]$:
\begin{align}
	&\begin{aligned}
			 \mathcal{A}_{\rho}^{m}=\big\{u_h: u_h\in C^{0}([0,T];W^{m,\infty}(\R^{3}))\ \forall\ h\in (0,1]&, \text{ and for all }t\in (0,T]\text{  and }h\in (0,1], \\ \|u_h\| _{L^\infty(\R^3)}\leq \rho,\quad  
			 &  \|\partial_ x^{\alpha }u_h\| _{L^\infty(\R^3)}\leq \rho h^{1-|\alpha|} \text{ if }1\leq |\alpha|\leq m\big\},\\
	\end{aligned}\\
	&\begin{aligned}
			 \mathcal{B}_{\rho}^{m}=\big\{u_h: u_h\in C^{0}([0,T];H^m(\R^{3}))\ \forall\ h\in (0,1]&, \text{ and for all }t\in (0,T)\text{  and }h\in (0,1], \\
			 &\quad   \|\partial_ x^{\alpha }u_h\| _{L^2(\R^3)}\leq \rho h^{-|\alpha|} \text{ for }0\leq |\alpha|\leq m\big\}.\\
	\end{aligned}
\end{align}

In our context, the theorem reads
\begin{theorem}[{\cite[Theorem 1.1]{Gues93}}]\label{thm:gues}
	Suppose that $m\geq 3$ and $M\geq m$, and let $\rho>0$ be given such that $\chi^{(3)}\in \mathcal{A}_\rho^{m}$. Then there exist $h_{\rho}$ and $\sigma_{\rho}$ such that if $v_h\in \mathcal{A}_\rho^{m+1}$ satisfies 
	\begin{equation}\label{eq:approximate_sol}
		L(\chi^{(3)}(x),v_{h})v_{h}\in h^{M}\mathcal{B}_{\rho}^{m},
	\end{equation}
	then for any given Cauchy data $g_h\in v_h\big|_{t=0}+h^{M}\mathcal{B}_{\rho}^{m}$, the Cauchy problem
	\begin{equation}
		L(\chi^{(3)}(x),u_{h})u_{{h}}=0,\quad u_{h}\big|_{t=0}=g_h
	\end{equation}
	admits for  $h\in (0,h_{\rho}]$ a unique solution $u_h\in v_h+h^{M}\mathcal{B}_{\sigma}^{m}$ on  $\Omega$. 
\end{theorem}

Fix $T>0$,  $m\geq 3$, and $M> m$
and let $\rho>0 $ be as in  the statement of Theorem \ref{thm:gues}, that is, with $\chi^{(3)}\in \mathcal{A}_\rho^{m}$.
Further, let $\u E_j$ for $0\leq j\leq M+1$ be as in Lemma~\ref{lm:stationary} with the $s$ there satisfying $s\geq m+1$.
We consider the expressions, defined for $t\in [0,T]$, 
\begin{equation}\label{eq:approx}
\begin{aligned}
		E_{{M}}(t,x)\coloneqq &h^{1/2}\underline{E}_0+\sum_{k=0}^{M} h^{k+3/	2}
	 \left(\underline{E}_{k+1}+U_{\mathrm{in},k}\Big(t,x,\frac{\phi_{\rm in}}{h}\Big)\right)
	 +\sum_{k=0}^{M} h^{k+3/	2}
	 U_{\mathrm{out},k}\Big(t,x,\frac{\phi_{\rm out}}{h}\Big),\\
	 H_{{M}}(t,x)\coloneqq&-\int_0^{t} \curl E_{{M}}(\tau,x,h)\d \tau,
\end{aligned}
\end{equation}
where $U_{\mathrm{in},k}=U_k$ as in \eqref{ansatz} satisfy the equations in Lemmas \ref{lm:eikonal_transport}-\ref{lm:C0}, and $U_{\rm{out},k}(\theta)=e^{i\theta}a_{\rm{in},+}^{(k)}+e^{-i\theta}a_{\rm{out},-}^{(k)}$, see \eqref{ain}.
Recalling that $\phi_{\rm in/out}=\mp t+x\cdot \omega$, we have 
\begin{equation}
 	\partial_t\left(U_{\mathrm{in/out},k}\big(t,x,\frac{\phi_{\mathrm{in/out}}}{h}\big) \right)
=\partial_t U_{\mathrm{in/out},k}\mp h^{-1}\partial_\theta U_{\mathrm{in/out},k}\Big|_{\big(t,x,\frac{\phi_{\mathrm{in/out}}}{h}\big)}.
 \end{equation}
Thus 
\begin{align}
		 \int_0^{t} \curl\Big( U_{\rm{in/out},k}\big(\tau,x,\frac{\phi_{\mathrm{in/out}}}{h}\big)\Big)\d \tau=&
		 \int_0^{t} \left(\curl U_{\rm{in/out},k}+ h^{-1}\omega\times  \partial_\theta U_{\rm{in/out},k}\right)\Big|_{\big(\tau,x, \frac{\phi_{\mathrm{in/out}}}{h}\big)} \d \tau\\\label{eq:approx_2}
		  =&
		 \int_0^{t} \left(\curl U_{\rm{in/out},k}\pm\omega\times  \partial_t U_{\mathrm{in/out},k}\right)\Big|_{\big(\tau,x, \frac{\phi_{\mathrm{in/out}}}{h}\big)} \d \tau \\* 
		 &\mp \omega \times U_{\mathrm{in/out},k}\Big(t,x,\frac{\phi_{\mathrm{in/out}}}{h}\Big)\pm \omega \times U_{\mathrm{in/out},k}\Big(0,x,\frac{x\cdot \omega}{h}\Big).
\end{align}
By the construction in Section \ref{sub:initialization_of_the_problem_the_linear_solution} and Proposition \ref{prop:approx_sol}, all of the $U_{\mathrm{in/out},k}$, $k\geq 0$, are compactly supported in space and smooth; the same holds for $\underline{E}_{k+1}$, while $\underline{E}_0$ is constant.
Upon setting  $v_M=(E_{{M}},H_{{M}})^{T}$, we can shrink $h$ enough to guarantee that $v_{M}\in \mathcal{A}_{\rho}^{m+1}$.
Further, we shrink $h$ even more if needed, so that $\chi^{(3)}|E_{{M}}|^2>-1/8$ (so that $\psi(\chi^{(3)}|E_{{M}}|^2)=1)$.

By Lemma \ref{lm:U1}, $E_{{M}}$ solves \eqref{eq:second_order_metric3} up to order $O(h^{M-1/2})$, so  we have
\begin{equation}
	\begin{aligned}
		 \p_t^2 E_{{M}}- \Delta E_{{M}} + \nabla \div E_{{M}}+\chi^{(3)}\p_t^2 (|E_{{M}}|^2E_{{M}}) 
		 =h^{M+1/2}(A_{\rm in}(t,x,\phi_{\rm in}/h)+A_{\rm out}(t,x,\phi_{\rm out}/h)),		
	\end{aligned}
\end{equation}
where $A_{\rm in/out}\in C_0^{\infty}([0,T]\times \R^{3}\times S^1)$, therefore if $h$ is small enough, the right hand side lies in $h^{M}\mathcal{B}_\rho^{m}$.
Hence by \eqref{eq:curl} and \eqref{eq:approx},\begin{equation}
	\partial_t( \partial_t (E_{{M}}+\chi^{(3)}|E_{{M}}|^2E_{{M}})-\curl H_{{M}})=h^{M+1/2}(A_{\rm in}(t,x,\phi_{\rm in}/h)+A_{\rm out}(t,x,\phi_{\rm out}/h)),
\end{equation}
and $\partial_t (E_{{M}}+\chi^{(3)}|E_{{M}}|^2E_{{M}})-\curl H_{{M}}\big|_{t=0}=0$,
thus in particular
\begin{equation}
	\partial_t (E_{{M}}+\chi^{(3)}|E_{{M}}|^2E_{{M}})-\curl H_{{M}}\in h^{M}\mathcal{B}_\rho^{m} , \qquad \partial_t H_{{M}}+\curl E_{{M}}=0.
\end{equation}
That is, \eqref{eq:approximate_sol} is fulfilled (recall $\psi(\chi^{(3)}|E_{{M}}|^2)=1$).
Regarding the Cauchy data of $v_M$,
\begin{equation}
	v_M\big|_{t=0}=\Big(\sum_{k=0}^{M+1}h^{k+1/2}\underline{E}_k+2h^{3/2}U_{\rm init}(x)\cos \left(\frac{x\cdot\omega}{h}\right),0\Big)^{T}.
\end{equation}
Using \eqref{eq:E_error}, and assuming that 	 $s>m+1$ there,  we see that the error compared to $g $ from \eqref{system_ic} satisfies $v_M\big|_{t=0}-g= O_{H^{s-1}(\R^{3})}(h^{M+5/2})$, so in particular in $h^{M}\mathcal{B}_\rho^{m}$, if $h$ is small enough.
We conclude that the assumptions of Theorem \ref{thm:gues} are satisfied, yielding the existence of a unique exact solution $u_h=(E,H)$ of \eqref{first_order_modified}-\eqref{system_ic} such that $u_h-v_M\in h^{M}\mathcal{B}_\sigma^{m}$, for some $\sigma>0$.
By the Sobolev embedding theorem (\cite[Theorem 3.26]{McLean-book})
\begin{equation}\label{eqSob_emb}
\sup_{\Omega }|u_h-v_M|\leq C\sup _{|\alpha|\leq m }\sup_{t\in [0,T]} \| \partial^{\alpha }(u_h-v_M)\|_{L^{2}(\R^3)}\leq C\sigma h^{M-m},	
\end{equation}
 so upon shrinking $h$ we can guarantee that $\chi^{(3)}|E|^2>-1/8$, which implies that $u_h$ solves the original system \eqref{eq1a}-\eqref{eq2a}.
Regarding the divergence-free conditions, 
they are satisfied by $u_h$ in $B(0,R_0)$ for $t=0$ by Lemma \ref{lm:stationary}.
Since $\partial_t \div (E+\chi^{(3)}|E|^2)=0$ and  $\partial_t \div H=0$ by \eqref{eq1a} and \eqref{eq2a} respectively, they are also satisfied for all $t\geq 0$ in $B(0,R_0)$.

Going back to \eqref{eq:second_order_metric3}, we would also like to show the existence of a unique solution in some class.
Let $m'\geq 4$,  and $M'\geq m'+2$,  which implies $M'-2>m'-1\geq 3.$
Then the original system \eqref{eq1a}-\eqref{eq2a} admits unique exact solutions $u_h=(E,H)\in v_{M'}+h^{M'}\mathcal{B}_\sigma^{m'}$ and  $u_h'=(E',H')\in v_{M'-2}+h^{M'-2}\mathcal{B}_{\sigma'}^{m'-1}$ on $\Omega$, for some $\sigma$, $\sigma'>0$.
We claim that they agree if $h$ is small enough.
Indeed, for $t\in [0,T]$, and for $|\alpha|\leq m'-1$,
\begin{align}
		 h^{2-M'}\| \partial^{\alpha}(u_h-v_{M'-2})\|_{L^2(\R^3)}\leq& h^{2-M'}\| \partial^{\alpha}(u_h-v_{M'})\|_{L^2(\R^3)}+h^{2-M'}\| \partial^{\alpha}(v_{M'}-v_{M'-2})\|_{L^2(\R^3)} 
\\ 
\leq& h^{2-M'}\| \partial^{\alpha}	(u_h-v_{M'})\|_{L^2(\R^3)}\\*
&+h^{2-M'}\| \partial^{\alpha}(h^{M'+1/2}(E_{M'}+U_{\mathrm{in},M'-1}+U_{\mathrm{out},M'-1}))\|_{L^2(\R^3)}\label{eq:difference}
\\*
&\quad +h^{2-M'}\| \partial^{\alpha}(h^{M'+3/2}(E_{M'+1}+U_{\mathrm{in},M'}+U_{\mathrm{out},M'}))\|_{L^2(\R^3)}\\
		 \leq& \sigma h^{-|\alpha|+2}+h^{2-M'}C_{\alpha}h^{M'+1/2-|\alpha|}\leq \sigma' h^{-|\alpha|}
\end{align}
for $h$ sufficiently small,
where for the third inequality we used that $v_{M'}-u_h\in h^{M'}\mathcal{B}_\sigma^{m'}$  as well as the fact that the quantities  in  \eqref{eq:approx} and  \eqref{eq:approx_2} are compactly supported and smooth.
We conclude that $u_h-v_{M'-2}\in h^{M'-2}\mathcal{B}_{\sigma'}^{m'-1}$ if $h$ is sufficiently small, so $u_h=u_h'$ by uniqueness.

Now consider the first part $E$ of the exact solution  $ u_h$, which
 solves \eqref{eq:second_order_metric3} with initial conditions \eqref{IC1}, and satisfies \eqref{div0} in $B(0,R_0)$.
We claim that this solution is unique in $E_{M'}+h^{M'}\mathcal{B}_\sigma^{m'}$, if $h$ is small enough.
It suffices to show that $H-H_{M'-2}\in h^{M'-2}\mathcal{B}_{\sigma'}^{m'-1}$, since then we obtain the claim by uniqueness of the solution to the system in $v_{M'-2}+h^{M'-2}\mathcal{B}_{\sigma'}^{m'-1}$.
We have, for $|\alpha|\leq m'-1$ 
\begin{align}
	h^{2-M'}	\| \partial^{\alpha}(H-H_{M'-2})\|_{L^2(\R^3)}
	\leq &\,h^{2-M'}	\| \partial^{\alpha}(H-H_{M'})\|_{L^2(\R^3)}+h^{2-M'}	\| \partial^{\alpha}(H_{M'}-H_{M'-2})\|_{L^2(\R^3)}\\ 
	{\leq} & h^{2-M'}	\| \int _0^{T}\partial^{\alpha}\curl(E-E_{M'}) \d t\|_{L^2(\R^3)}+h^{2-M'}	C_\alpha h^{M'+1/2-|\alpha|}\\ 
		\leq &h^{2-M'}	 \int _0^{T}\|\partial^{\alpha}\curl (E-E_{M'})\|_{L^2(\R^3)} \d t+	C_\alpha h^{5/2-|\alpha|}\\ 
		\leq &\,	h^{2-M'} \int _0^{T} \sigma h^{M'-|\alpha|-1}\d t+	C_\alpha h^{5/2-|\alpha|}\\ 
		\leq&\, (hT\sigma+C_\alpha h^{5/2}) h^{-|\alpha|}\leq \sigma' h^{-|\alpha|},
\end{align}
provided $ h$ is sufficiently small.
For the second inequality we used \eqref{eq:approx_2} and a computation as in \eqref{eq:difference};
for the third, Minkowski's integral inequality, and for the fourth the fact that $E\in E_{M'}+h^{M'}\mathcal{B}_\sigma^{m'}$, so that
$h^{-M'}\|\partial^{\beta}(E-E_{M'})\|_{L^2(\R^3)}\leq \sigma h^{-|\beta|}$ for all $t\in [0,T]$ and $|\beta|\leq m'$.
Since we just showed that $E\in E_{M'}+h^{M'}\mathcal{B}_\sigma^{m'}$  implies $(E,H) \in v_{M'-2}+h^{M'-2}\mathcal{B}_{\sigma'}^{m'-1}$, and we have uniqueness for solutions of the system in the latter space, $E$ is unique in the former.

We have reached the following proposition, which proves Theorem~\ref{thm_main}\ref{item:b} and \ref{item:c}. Note that \ref{item:b} is a special case of \ref{item:c}, when $\underline{E}=0$.

\begin{proposition} \label{prop_Gues}
	Let $T>0$, and $m\geq 3$ and $M> m$ be integers.
	Consider  the  asymptotic solution $v_M=(E_{{M}},H_{{M}})^{T}$, see \eqref{eq:approx},	defined for $t\in [0,T]$.
	There exists $\sigma>0$ and $h_0$ such that for $0<h\leq h_0$,   the system \eqref{eq1a}-\eqref{eq2a} with initial conditions \eqref{system_ic}, where $\underline{E}$ is as in Lemma~\ref{lm:stationary} with $s\geq m+1$, has a unique 	solution  $(E,H)\in v_M+ h^{M}\mathcal{B}_{\sigma}^{m}$ on $\Omega$. Moreover, \eqref{eq3a}-\eqref{eq4a} are satisfied in a large ball of radius $R_0$.
	If we take  $m\geq 4$ and $M\geq m+2$, then $E$ above is the unique solution to \eqref{eq:second_order_metric3}, \eqref{IC1} in $E_h+h^{M}\mathcal{B}^{m}_\sigma$, and it also solves \eqref{div0} for $|x|	<R_0$ and $t\in [0,T]$.
\end{proposition}

\section{Proof of Theorem~\ref{thm_inverse}} \label{sec_inverse}
\begin{proof}[Proof of Theorem~\ref{thm_inverse}]
We construct a solution $E$ as in Theorem \ref{thm_main}\ref{item:c}, using the more precise statement of Proposition \ref{prop_Gues}.
Spefically, we let  $m\geq 4 $ and $M\geq m+3$ there to obtain uniqueness, as well as to guarantee that if  $f\in h^{M-3/2}\mathcal{B}^{m}_\sigma $ then $ h^{-1}\|f\|_{L^\infty} \leq C  $ for some $h$-independent constant (this follows as in \eqref{eqSob_emb}).
As before, we assume $c=1$. 

First, knowing $E_\textrm{in}$ for $x\in \pi_{R,\omega}$ as stated in the theorem, and $\omega$ fixed, we explain now how to use a similar argument as in \cite[Proposition~3.2]{S-Antonio-nonlinear} to separate the high-frequency part of $E_\textrm{in}$ in \r{exp1} from the rest, up to $O(h^{5/2})$.
To this end, assume for now that $\omega = e_1$, write $x=(x_1,x')$, and  choose  $\psi\in C_0^\infty((0,2R_0))$. 
Then (after subtracting $h^{1/2}\u E_0$), multiply $ E_{\rm in }$ by  $\psi(t)\cos(\phi/h)$
% , also by $\psi(t)\sin(\phi/h)$, 
and integrate in $t$ over a time interval equal to $\supp U_\textrm{init}(\cdot, x)$.  
For $x'$ so that $|x'|\le R$, %
we have for the second component of $E_{\rm in}$ 
\begin{align}
\int_0^{2R_0}& h^{-3/2} E_\textrm{in,2}(t,R,x') \cos\frac{-t+R}{h}\psi(t)\,\d t\\ &= \int_0^{2R_0} U_{{\rm init},2}(R-t,x') \cos\Big(\frac{-t+R}{h}+\tau(R,x')\Big) \cos\frac{-t+R}{h}\psi(t)\,\d t + O(h) \\
&= \frac12 \cos( \tau(R,x')) \int _0^{2R_0}U_{{\rm init},2}(R-t,x') \psi(t)\,\d t + O_{L^\infty}(h), \label{eq:inv1}
\end{align}
using the product of cosines formula and a change of variables to absorb the second term generated by the latter into the error.
To see that the error is $O_{L^\infty}(h)$, notice that 
$E$ differs from an asymptotic solution containing finitely many terms by an element of  $h^{M}\mathcal{B}_{\sigma}^{m}$  for some $\sigma>0$, and the term corresponding to each $k\geq 0$ in \eqref{eq:approx} is in $ h^{k+3/2}L^\infty$.
Since the total error in \eqref{eq:inv1}  includes terms corresponding to $k\geq 1$ only, it is in $hL^\infty+h^{M-3/2}\mathcal{B}_{\sigma}^{m}\subset hL^\infty$ by our choice of $M$. 
Since $ U_{{\rm init},2}(\cdot,x')\not\equiv 0$, the integral in \eqref{eq:inv1}   can be made non-vanishing with a suitable $\psi$ in a neighborhood of every $x'$ with $|x'|\le R$, and eventually we can use a partition of unity. This recovers $\cos( \tau(R,x'))$ uniquely. Similarly, we recover $\sin( \tau(R,x'))$ by multiplying by $\psi(t)\sin(\phi/h)$ instead of $\psi(t)\cos(\phi/h)$. 

This process recovers $\tau(R,x')$ up to $2\pi k$, $k$ integer. Then if we have two such $\tau_j$'s corresponding to two nonlinearity coefficients, we  get $\tau_1(R,x')-\tau_2(R,x')=2\pi k(x')$ on the ``detector'', see Figure~\ref{fig:the_setup}. Since the left-hand size is smooth and compactly supported, we must have $k(x')=0$. By \r{tau}, this recovers the X-ray transform of $\chi^{(3)}$ in the direction $\omega$. As stated in the theorem, varying $\omega$, we recover $\chi^{(3)}$ itself. 
\end{proof}

Note that in the proof above, we projected to $e_2=(0,1,0)$ only. We could have projected the field to $e_3$ as well. Physically, the separation of the high from the low frequencies above corresponds to measuring light vs.\ a constant electric field, and the projection to $e_2$ of $e_3$ corresponds to applying polarizing filters.

\begin{remark}\label{remark_cell}
Assume now that $ U_{\rm init}(x_1,x')=(0,a_2,a_3)=\textrm{const.}$ for $|x'|\le R$, $|x_1|<\delta $ with some $\delta>0$. Then  we can project $E_\textrm{in}(t,R,x')$ to $(0,-a_3,a_2)/\sqrt{a_2^2+ a_3^2}$ as in Section \ref{sec_cell}, and with $\psi$ supported in $(R-\delta,R+\delta)$ we can extract the high-frequency behavior as in \r{eq:inv1}. In this way we justify  the arguments there, where we analyzed the leading high-frequency terms only. 
\end{remark}

\section{Concluding remarks} \label{sec_scale}
We make some remarks on the magnitudes of the physical quantities involved. 
The speed of light in vacuum is  $c_0 \approx 3\times 10^8$ m/sec, the permittivity of free space is $\epsilon_0\approx 8.854\times 10^{-12}$ F/m, and $\chi^{(3)}$ of nitrobenzene is of the order of $\chi^{(3)}\approx 5.7\times 10^{-20}$ $\text{m}^2/\text{V}^2$ (see \cite[Table 4.1.2]{boyd2020nonlinear}). 
 A typical external electric field is of the order of $\tilde{\underline E} \sim 30$ kV/cm, where $\tilde{\underline E}=h^{1/2}\underline{E} $ in our notation. A typical wavelength (for green) is   $\lambda\approx 5\times 10^{-7}$ m. The scaling we chose requires $\chi^{(3)}|\tilde{ \underline{E}_0}|^2\sim h$ with $\lambda=2\pi h$. With those values, we get $\chi^{(3)}|\tilde E_0|^2/\lambda\approx 1$, which is consistent with our scaling.  

This shows that $\chi^{(3)}$ is very small in the metric units and remains small even relative to the wavelength $\lambda$, which is $2\pi h$ in our setup. We treat $\chi^{(3)}$  implicitly as $\sim 1$. On the other hand, the strong electric field $\underline E$ has an amplitude of the order $h^{1/2}$, which is not ``strong.'' We can of course rescale those quantities. We can take $ E^\sharp = h^{1/2-p}\underline{E} + h^{3/2-p}U$. Choosing a suitable $p>1/2$, we get closer to physics, where $\chi^{(3)}$, now replaced by $h^{2p} \chi^{(3)}$ is indeed small, and $h^{1/2}\underline E$, now rescaled to $h^{-p+1/2}\underline E$, is large. The requirement $\chi^{(3)}|E|^2\sim h$ is preserved. 

It may seem non-physical to relate the ratio of the amplitude of the beam to that of the strong electric field, which is a dimensionless variable, to the wavelength $2\pi h$, which has units of length. In fact,  we are interested in propagation over lengths $L=T$ over the time interval $[0,T]$ (since $c=1$ after rescaling $t$), a typical one being $10$ cm. Then $2\pi h$ should be considered as the wavelength relative to  $L$, i.e., as their ratio, which is dimensionless. Finally, the typical beams involved in the DC Kerr effect are not necessarily of small width like laser rays, even though some are. For this reason, we do not treat that width, which in our case would be the size of $\supp U_\textrm{init}$ in directions perpendicular to $\omega$, as a small parameter.

The prefactor $h^{3/2}$ in \eqref{A1} multiplying $U_{\rm in}$ is motivated by our need to extend the asymptotic construction to any order $O(h^N)$. Since the size of $U_\textrm{init}$ in \eqref{IC1} is arbitrary, as long as it is $h$-independent, it can be ``small'' if we want to model even weaker beams. It can be large as well, but then we would have to take $h$ even smaller for similar error bounds. On  the other hand, if we put  $h^{1/2}\epsilon U$ instead of $h^{3/2} U$ in  \eqref{A1}, with $0<\epsilon\ll1$, the leading profile equation would be the same but one would have trouble deriving the lower order terms for an asymptotic solution, and ultimately, proving existence of an exact solution close to it. 
Trying to take $h^{1/2} U_{\rm in}$  there (instead of $h^{3/2}U_{\rm in}$) puts us in the strongly nonlinear regime, which changes the geometry of the wave propagation, and is much harder to justify. In other words, we really want the beam to be weaker than the strong electric field, and it being weaker by the factor $h$ is what makes the construction work. 

We would like to point out possible generalizations which we believe can be treated in the same way. In Theorem~\ref{thm_main}\ref{item:c}, one can have a strong electric field $\underline{E}$ with a  not necessarily constant leading term $h^{1/2}\underline E_0$ as in part (a). 
Also, $\underline{E}$ could be $t$ (and $x$) dependent as long as it does not oscillate at the rate $1/h$. Finally, $\chi^{(1)}>-1$ could be nontrivial and $x$-dependent, as long as the beam does not develop caustics over the time interval of interest, but we expect the treatment of this case to complicate the exposition substantially.  

%
%
%
%
%
%
%
%
%
%
%
%
% %

% %
% %

%%\printbibliography
% \bibliographystyle{abbrv}
%% %
% \bibliography{References.bib}
%

\end{document}